\documentclass[11pt,a4paper]{article}
\usepackage{amssymb}
\usepackage{amsmath}
\usepackage{amsfonts}
\usepackage{bbm}
\usepackage{amsthm}
\usepackage{mathrsfs}
\usepackage{hyperref}
\usepackage{color}
\usepackage[utf8]{inputenc}
\usepackage{stmaryrd}
\usepackage{rotating}
\usepackage{xfrac}

\definecolor{darkred}{rgb}{0.8,0.1,0.1}
\hypersetup{
     colorlinks=false,         % false: boxed links; true: colored links,false is default
     linkcolor=darkred,
     citecolor=blue,
}

\usepackage{comment}
\usepackage[margin=2.6cm]{geometry}
\usepackage[all,cmtip]{xy}

\usepackage{marvosym}

\theoremstyle{plain}
\newtheorem{theo}{Theorem}[section]
\newtheorem{lem}[theo]{Lemma}
\newtheorem{propo}[theo]{Proposition}
\newtheorem{cor}[theo]{Corollary}

\theoremstyle{definition}
\newtheorem{defi}[theo]{Definition}
\newtheorem{ex}[theo]{Example}

\newtheorem{rem}[theo]{Remark}

\numberwithin{equation}{section}

\def\nn{\nonumber}

\def\Hom{\mathrm{Hom}}
\def\hom{\mathrm{hom}}

\def\dd{\mathrm{d}}
\def\op{\mathrm{op}}
\def\id{\mathrm{id}}
\def\ev{\mathrm{ev}}
\def\MMM{\mathscr{M}}
\def\AAA{\mathscr{A}}
\def\FF{\mathcal{F}}
\def\ra{\triangleright}

\def\bbZ{\mathbb{Z}}

\def\LLL{\mathscr{L}}
\def\RRR{\mathscr{R}}
\def\DDD{\mathrm{ad}_{\bullet}}

\newcommand{\obultimes}{\mathbin{\ooalign{$\otimes$\cr\hidewidth\raise0.17ex\hbox{$\scriptstyle\bullet\mkern4.48mu$}}}}

\newcommand{\obulplus}{\mathbin{\ooalign{$\boxplus$\cr\hidewidth\raise0.295ex\hbox{$\scriptstyle\bullet\mkern4.7mu$}}}}

\newcommand{\bol}[1]{#1}
\newcommand{\udl}[1]{\underline{#1}}

\def\sk{\vspace{2mm}}

\title{%
\Large Nonassociative geometry in quasi-Hopf representation categories II:\\[1mm]
Connections and curvature
}

\author{%
Gwendolyn E. Barnes$^{a}$, Alexander Schenkel$^{b}$ \ and \ Richard J.\ Szabo$^{c}$\vspace{4mm}\\
{\small Department of Mathematics, Heriot-Watt University,}\\
{\small Colin Maclaurin Building, Riccarton, Edinburgh EH14 4AS, United Kingdom}\vspace{1mm}\\
{\small Maxwell Institute for Mathematical Sciences, Edinburgh, United Kingdom}\vspace{1mm}\\
{\small The Higgs Centre for Theoretical Physics, Edinburgh, United Kingdom}\vspace{4mm}\\
 {\footnotesize \texttt{email:} $^a$ \texttt{geb31@hw.ac.uk} ~,~~$^b$  \texttt{as880@hw.ac.uk} ~,~~$^c$ \texttt{R.J.Szabo@hw.ac.uk} }
 }

\date{July 2015}

\begin{document}

\maketitle

\begin{abstract}
\noindent
We continue our systematic development of noncommutative and
nonassociative differential geometry internal to the representation
category of a quasitriangular quasi-Hopf algebra. We describe
derivations, differential operators, differential calculi and
connections using universal categorical constructions to capture
algebraic properties such as Leibniz rules. Our main result is the
construction of morphisms which provide prescriptions for lifting
connections to tensor products and to internal homomorphisms. We describe
the curvatures of connections within our formalism, and also the
formulation of Einstein-Cartan geometry as a putative framework for a
nonassociative theory of gravity.
\end{abstract}

\paragraph*{Report no.:} EMPG--15--{10}
braided monoidal categories, internal homomorphisms, cochain twist quantization
\paragraph*{MSC 2010:} 16T05, 17B37, 46L87, 53D55

%%%%%%%%%%%%%%%%%%%%%%%%%%%%%%%%%%%%%%%%%%%%%%%%%%%%%%%
%%%%%%%%%%%%%%%%%%%%%%%%%%%%%%%%%%%%%%%%%%%%%%%%%%%%%%%
%\newpage

%\setcounter{tocdepth}{1}
\tableofcontents

\section{\label{sec:intro}Introduction and summary}
This paper is the second part in a series of articles whose goal is to systematically develop
a formalism for differential geometry on noncommutative and
nonassociative spaces. The main physical inspiration behind this work
is sparked by the recent observations from closed string theory that
certain non-geometric flux compactifications experience a
nonassociative deformation of the spacetime
geometry~\cite{Bouwknegt:2004ap,Blumenhagen:2010hj,Lust:2010iy,Blumenhagen:2011ph,Mylonas:2012pg,Blumenhagen:2013zpa}
(see~\cite{Lust:2012fp,Mylonasa,Blumenhagen:2014sba} for reviews and
further references), together with the constructions
of~\cite{Mylonasb,Aschieri:2015roa} which show that the corresponding
nonassociative algebras and their basic geometric structures can be
obtained by cochain twist quantization, and hence are commutative and associative
quantities when regarded as objects in a suitable braided monoidal
category. See the first paper in this series~\cite{Barnes:2014},
hereafter referred to as Part~I, for
further motivation and a more complete list of relevant references.
\sk

Earlier categorical approaches to nonassociative geometry along these lines were pursued in~\cite{Bouwknegt:2007sk,BeggsMajid1}. In the present paper we develop important notions of 
differential geometry internal to the representation category ${}^H\MMM$ of a quasitriangular 
quasi-Hopf algebra $H$. In particular, we develop the notions of derivations, differential operators, differential calculi and 
connections by using universal categorical constructions such as categorical limits. 
In contrast to the approach of~\cite{BeggsMajid1}, our geometric structures
are described by internal homomorphisms instead of morphisms in the category ${}^H\MMM$.
This leads to a much richer framework, because the conditions for being a morphism
in ${}^H\MMM$ (i.e.\ $H$-equivariance) are very restrictive and hence the
framework in~\cite{BeggsMajid1} allows for only very special geometric structures.
Our internal homomorphism approach is inspired by the formalism of~\cite{AschieriSchenkel} (see~\cite{Schenkel:2011ae,Aschieri:2012af} for overviews),
and it clarifies these ideas and constructions in categorical terms.
\sk

We begin in Section~\ref{sec:gradedrepresenttions} with a brief review 
of the categorical framework which was developed in Part~I.
In contrast to that paper, in the present paper we consider the case where all
modules are $\bbZ$-graded; this allows us later on to regard graded objects
such as differential calculi naturally as objects in our categories.
\sk

In Section~\ref{sec:derivations} we introduce derivations $\mathrm{der}(A)$ on braided commutative algebras $A$
in ${}^H\MMM$ by formalizing the Leibniz rule in terms of an equalizer in ${}^H\MMM$. We analyse structural properties of $\mathrm{der}(A)$ and in particular prove that, in the case
where $H$ is triangular, $\mathrm{der}(A)$ together with an internal commutator $[\,\cdot\, ,\,\cdot\, ]$
is a Lie algebra in ${}^H\MMM$. We then introduce differential operators
$\mathrm{diff}(V)$ on symmetric $A$-bimodules $V$ in ${}^H\MMM$ by again using a suitable
equalizer in ${}^H\MMM$ to capture the relevant algebraic properties. We show that
$\mathrm{diff}(V)$ is an algebra in ${}^H\MMM$ and we also
prove that the zeroth order differential operators are the internal endomorphisms $\mathrm{end}_A(V)$
in the category of symmetric $A$-bimodules ${}^H_{}{}^{}_{A}\MMM^{\mathrm{sym}}_{A}$.
Using the product structure on differential operators to formalize nilpotency of
a differential, we can then give a definition of a differential calculus in ${}^H\MMM$.
\sk

In Section~\ref{sec:connections} we develop an appropriate notion of connections $\mathrm{con}(V)$ 
on objects $V$ in ${}^{H}_{}{}^{}_{\bol{A}}\MMM_{\bol{A}}^{\mathrm{sym}}$. The idea is
to formalize a generalization of the usual Leibniz rule with respect to a differential calculus
in terms of an equalizer in ${}^H\MMM$. The resulting object
$\mathrm{con}(V)$ is analysed in detail and it is shown that the usual affine space
of ordinary connections arises as a certain proper subset of  $\mathrm{con}(V)$. Our more
flexible definition of connections has the advantage that $\mathrm{con}(V)$
also forms an object in ${}^H\MMM$ in addition to being an affine space.
We then develop a lifting prescription for connections to tensor products $V\otimes_A W$
of objects $V,W$ in ${}^{H}_{}{}^{}_{\bol{A}}\MMM_{\bol{A}}^{\mathrm{sym}}$. 
It is important to notice that our notion of tensor product connections differs from the
standard one: Although our techniques are only applicable to
braided commutative algebras and their bimodules in ${}^H\MMM$, they are more
flexible in the sense that {\em any} two connections can be lifted to a tensor product connection, not only those which satisfy the very restrictive `bimodule connection' property proposed in~\cite{Mourad,DuViMasson,Bresser,DuViLecture}.
We also develop a lifting prescription for connections to internal homomorphisms
$\hom_A(V,W)$ of objects $V,W$ in
${}^{H}_{}{}^{}_{\bol{A}}\MMM_{\bol{A}}^{\mathrm{sym}}$. These lifts
are all important ingredients in (noncommutative and nonassociative) Riemannian geometry
for extending e.g.\ tangent bundle connections to all tensor fields,
and they play an instrumental role in physical applications of our
formalism to noncommutative and nonassociative gravity theories such
as those anticipated to arise in non-geometric string theory. All of these
constructions moreover generalize and clarify the corresponding constructions of~\cite{AschieriSchenkel} 
in categorical terms. 
\sk

Finally, in Section~\ref{sec:curvature} we assign curvatures to connections and show that
they are internal endomorphisms in the category ${}^{H}_{}{}^{}_{\bol{A}}\MMM_{\bol{A}}^{\mathrm{sym}}$,
provided that $H$ is triangular. We also obtain a Bianchi tensor,
which in classical differential geometry would identically vanish; in
general it is not necessarily equal to $0$,
and hence in this sense it
characterizes the noncommutativity and nonassociativity of our geometries.
We further observe that the curvature of any tensor product connection is the sum of the two individual curvatures,
which means that curvatures behave additively in an appropriate sense.
We conclude with a brief outline of how our formalism
could be used to describe a noncommutative and nonassociative theory of
gravity coupled to Dirac fields; our considerations are based on Einstein-Cartan geometry
and its noncommutative generalization which was developed in~\cite{Catellani:2009}.

%%%%%%%%%%%%%%%%%%%%%%%%%%%%%%%%%%%%%%%%%%%%%%%%%%%%%%%
%%%%%%%%%%%%%%%%%%%%%%%%%%%%%%%%%%%%%%%%%%%%%%%%%%%%%%%

\section{\label{sec:gradedrepresenttions}Categorical preliminaries}
Let $k$ be an associative and commutative ring with unit $1\in k$. In
contrast to Part~I,
in this paper we shall work with $\bbZ$-graded $k$-modules.
This will have the advantage later on that naturally graded objects such as differential calculi can be described as objects 
in the categories we define below, and also that minus signs will be absorbed into the formalism.
The goal of this section is to adapt the material developed in~\cite{Barnes:2014} to the graded setting
and to thereby also fix our notation for the present paper.

%%%%%%%%%%%%%%%%%%%%%%%%%%%%%%%%%%%%%%%%%%%%%%%%%%%%%%%
%%%%%%%%%%%%%%%%%%%%%%%%%%%%%%%%%%%%%%%%%%%%%%%%%%%%%%%

\subsection{$\bbZ$-graded $k$-modules}
The category $\MMM^{}$ of bounded $\bbZ$-graded $k$-modules is defined as follows:
The objects in $\MMM$ are the bounded $\bbZ$-graded $k$-modules
\begin{flalign}
\udl{V} = \bigoplus_{n \in \bbZ}\, \udl{V}_{\,n}~,
\end{flalign}
where the $k$-modules $\udl{V}_{\,n} = 0$ for all but finitely many $n$. 
The morphisms in $\MMM$ are the degree preserving $k$-linear maps $f : \udl{V}\to \udl{W}$, i.e.\
$f(\udl{V}_{\,n}) \subseteq \udl{W}_{\,n}$ for all $n\in \bbZ$. For any object
$\udl{V}$ in $\MMM$ there is a map
\begin{flalign}
\vert \cdot \vert: \bigsqcup_{n \in \bbZ}\, \udl{V}_{\,n} \longrightarrow \bbZ~,
\end{flalign}
which assigns to elements $v\in \udl{V}_{\,n}$ their degree $\vert
v\vert =n$. Elements of $\udl{V}_{\,n}$ are said to be homogeneous of degree $n$.
\sk

The category $\MMM$ is monoidal with monoidal functor $\otimes : \MMM\times \MMM \to \MMM$
given by the $\bbZ$-graded tensor product: For any two objects $\udl{V},\udl{W}$ in $\MMM$ we define
\begin{flalign}\label{eqn:tensorgraded}
\udl{V} \otimes^{} \udl{W} := \bigoplus_{n \in \bbZ}\, \big(\udl{V} \otimes^{} \udl{W}\big)_{n} 
:= \bigoplus_{n\in\bbZ} \ \bigoplus_{m + l = n} \, \udl{V}_{\,m} \otimes^{} \udl{W}_{\,l}~,
\end{flalign}
where $\udl{V}_{\,m} \otimes^{} \udl{W}_{\,l}$ is the usual tensor product of $k$-modules. 
To any $\MMM^{} \times \MMM^{}$-morphism $(f: \udl{V} \to \udl{V}^\prime,~ g: \udl{W} \to \udl{W}^\prime\, )$ 
the monoidal functor assigns the $\MMM^{}$-morphism 
\begin{flalign}\label{eqn:tensormorphgraded}
f \otimes^{} g: \udl{V} \otimes^{} \udl{W}
\longrightarrow \udl{V}^\prime \otimes^{} \udl{W}^\prime ~, \qquad v \otimes^{} w\longmapsto  f(v) \otimes^{} g(w)~.
\end{flalign}
The unit object in $\MMM$ is
given by the ring $k$ itself, but regarded as a $\bbZ$-graded $k$-module with 
$k_n =0$, for all $n\neq 0$, and $k_0 = k$.
The associator in $\MMM^{}$ is the natural isomorphism 
\begin{flalign}
\udl{\Phi}^{}: \otimes^{} \circ \big(\otimes^{}\times \id^{}_{\MMM^{}}\big) 
\Longrightarrow \otimes^{}\circ \big(\id^{}_{\MMM^{}}\times \otimes^{}\big)~,
\end{flalign}
whose components are identity maps. The unitors in $\MMM^{}$ are the natural isomorphisms 
\begin{flalign}
\udl{\lambda} : k\otimes^{} \text{--} \Longrightarrow \id_{\MMM^{}}
\qquad \mbox{and} \qquad \udl{\rho} : \text{--}\otimes^{} k \Longrightarrow \id_{\MMM^{}}~,
\end{flalign}
whose components are 
$\udl{\lambda} : k\otimes^{} \udl{V}\to \udl{V} \,,~ c \otimes^{} v\mapsto c \,v$ and 
$\udl{\rho} : \udl{V}\otimes^{} k \to \udl{V}\,,~v\otimes^{} c \mapsto c\,v$. 
Here and in the following we shall refrain from writing indices on the 
components of natural transformations.
\sk

The monoidal category $\MMM$ is also braided.
Denoting by $\otimes^\op : \MMM\times \MMM\to \MMM$ the functor
taking the opposite $\bbZ$-graded tensor product, i.e.\ $\udl{V} \otimes^{\op} \udl{W} := \udl{W}\otimes \udl{V}$
and similarly for morphisms, the braiding is the natural isomorphism
\begin{flalign}
\udl{\tau} : \otimes \Longrightarrow \otimes^{\op}~,
\end{flalign}
whose components are given by
\begin{flalign}
\udl{\tau} : \udl{V}\otimes \udl{W}\longrightarrow \udl{W}\otimes \udl{V}~, \qquad
v\otimes w \longmapsto (-1)^{\vert v\vert\,\vert w\vert} ~w\otimes v~,
\end{flalign}
for all homogeneous $v\in \udl{V}$ and $w\in \udl{W}$. 
\sk

Finally, $\MMM^{}$ is a braided closed monoidal category with
internal $\hom$-functor 
\begin{flalign}
\udl{\hom}^{}: \MMM^\op \times \MMM^{} \longrightarrow \MMM^{}~.
\end{flalign}
For any object $(\udl{V},\udl{W})$ in $\MMM^\op \times \MMM^{}$ we define
\begin{flalign}\label{eqn:internalhomgraded}
\udl{\hom}^{}\big(\udl{V},\udl{W}\big) := \bigoplus_{n \in \bbZ} \, \udl{\hom}\big(\udl{V},\udl{W}\big)_{n}
:=\bigoplus_{n\in\bbZ} \  \bigoplus_{l-m=n} \, \Hom_k\big(\udl{V}_{\,m},\udl{W}_{\,l}\big) ~,
\end{flalign}
where $\Hom_k (\udl{V}_{\,m},\udl{W}_{\,l} )$ denotes the $k$-module of $k$-linear 
maps between the homogeneous components $ \udl{V}_{\,m}$ and $\udl{W}_{\,l} $.
To any $ \MMM^\op \times \MMM^{} $-morphism 
$ (f^\op:\udl{V} \to \udl{V}^\prime, g:\udl{W} \to \udl{W}^\prime\,) $ the internal $\hom$-functor assigns the
$ \MMM^{} $-morphism
\begin{flalign}\label{eqn:internalhommorph}
\udl{\hom}^{}(f^\op,g): \udl{\hom}^{}\big(\udl{V},\udl{W}\big) \longrightarrow 
\udl{\hom}^{}\big(\udl{V}^\prime,\udl{W}^\prime\, \big)~, \qquad L \longmapsto g \circ L \circ f~.
\end{flalign}
The natural currying isomorphism
\begin{flalign}
\udl{\zeta} : \Hom_{\MMM}^{}(\text{--}\otimes\text{--},\text{--}) \Longrightarrow 
\Hom_{\MMM}^{}(\text{--},\udl{\hom}^{}(\text{--},\text{--}))
\end{flalign}
has components given by
\begin{flalign}
\udl{\zeta}(f) : \udl{V}\longrightarrow \udl{\hom}^{}(\udl{W},\udl{X})~, \qquad v\longmapsto f\big(v\otimes (\,\cdot\,)\big)~,
\end{flalign}
for all $\MMM$-morphisms $f : \udl{V}\otimes \udl{W}\to \udl{X}$. The natural inverse currying isomorphism has
components given by
\begin{flalign}
\udl{\zeta}^{-1}(g) : \udl{V}\otimes \udl{W}\longrightarrow \udl{X}~,\qquad v\otimes w\longmapsto g(v)(w)~,
\end{flalign}
for all $\MMM$-morphisms $g : \udl{V}\to \udl{\hom}(\udl{W},\udl{X})$.

%%%%%%%%%%%%%%%%%%%%%%%%%%%%%%%%%%%%%%%%%%%%%%%%%%%%%%%
%%%%%%%%%%%%%%%%%%%%%%%%%%%%%%%%%%%%%%%%%%%%%%%%%%%%%%%

\subsection{$\bbZ$-graded quasi-Hopf representation categories}
Let $H$ be a quasitriangular quasi-Hopf algebra (see e.g.\
\cite{Drinfeld,Barnes:2014} for definitions), which we regard as being $\bbZ$-graded and sitting in degree $0$.
A left $H$-action on an object $\udl{V}$ in $\MMM$ is an $\MMM^{}$-morphism
\begin{flalign}
\ra : H \otimes^{} \udl{V} \longrightarrow \udl{V} ~,
\end{flalign}
such that
\begin{flalign}
1 \ra v = v \qquad \mbox{and} \qquad (h \,h^\prime\,) \ra v = h \ra (h^\prime \ra v)~,
\end{flalign}
for all $v \in \udl{V}$ and $h, h^\prime \in H$. 
The bounded $\bbZ$-graded representation category $^{H}\MMM^{}$ of $H$ 
is defined as follows: The objects in $^{H}\MMM^{}$ are the
bounded $\bbZ$-graded left $H$-modules, i.e.\ the pairs $V = (\udl{V}, \ra)$ consisting of an object $\udl{V}$ in $\MMM$ 
and a left $H$-action $\ra : H\otimes \udl{V}\to\udl{V}$
on $\udl{V}$. The morphisms in ${}^H\MMM$ are the $H$-equivariant $\MMM^{}$-morphisms $f: V\to W$, i.e.\
\begin{flalign}\label{eqn:Hequivariance}
f(h\ra v) = h\ra f(v)~,
\end{flalign}
for all $h\in H$ and $v\in \udl{V}$.
\sk

The category ${}^H\MMM$ is monoidal with monoidal functor
$\otimes^{}: {}^H_{}\MMM^{} \times {}^H_{}\MMM^{} \rightarrow {}^H_{}\MMM^{}$
(denoted by the same symbol as that of $\MMM$): For any two objects $V$,$W$ in ${}^H\MMM$ we define $V\otimes W$ as the 
$\bbZ$-graded tensor product of the underlying $\bbZ$-graded $k$-modules (\ref{eqn:tensorgraded})
together with the left $H$-action
\begin{flalign}
\ra: H \otimes \udl{V} \otimes \udl{W} \longrightarrow  \udl{V} \otimes \udl{W} ~, 
\qquad h \otimes v \otimes w \longmapsto (h_{(1)} \ra v) \otimes (h_{(2)} \ra w)~,
\end{flalign}
where we have used the Sweedler notation $\Delta(h) = h_{(1)}\otimes h_{(2)}$ (with summation understood)
for the coproduct of $H$.
To any ${}^H\MMM^{} \times {}^{H}\MMM^{}$-morphism $(f: V \to V^\prime,~ g: W\to W^\prime\,)$ 
the monoidal functor assigns the ${}^H\MMM^{}$-morphism induced by (\ref{eqn:tensormorphgraded}).
The unit object in ${}^H\MMM$ is $I := (k,\ra)$ with trivial left $H$-action
$\ra : H\otimes k\to k\,,~h\otimes c \mapsto \epsilon(h)\,c$ given by the counit of $H$.
The associator in ${}^H_{}\MMM^{}$ is the natural isomorphism  
\begin{flalign}
\Phi^{}: \otimes^{} \circ \big(\otimes^{}\times \id^{}_{{}^H\MMM^{}}\big) 
\Longrightarrow \otimes^{}\circ \big(\id^{}_{{}^{H}\MMM^{}}\times \otimes^{}\big)~,
\end{flalign}
whose components are given in terms of the 
associator $\phi = \phi^{(1)}\otimes\phi^{(2)}\otimes\phi^{(3)}\in H\otimes H\otimes H$ of $H$ by
 \begin{flalign}\label{eqn:associator}
 \nn \Phi^{} : (\bol{V}\otimes\bol{W})\otimes\bol{X}  &\longrightarrow \bol{V}\otimes(\bol{W}\otimes\bol{X})~,\\
  (v\otimes w) \otimes x& \longmapsto (\phi^{(1)}\ra v)\otimes \big((\phi^{(2)}\ra w)\otimes (\phi^{(3)}\ra x)\big)~.
 \end{flalign}
The unitors in ${}^H\MMM^{}$ are the natural isomorphisms 
\begin{flalign}
\lambda: I\otimes^{} \text{--} \Longrightarrow \id_{{}^H\MMM^{}} \qquad \mbox{and} \qquad
\rho : \text{--}\otimes^{} I \Longrightarrow \id_{{}^H\MMM^{}}~,
\end{flalign}
whose components are 
$\lambda : I\otimes^{} V\to V \,,~ c \otimes^{} v\mapsto c \,v$ and 
$\rho: V\otimes^{} I \to V\,,~v\otimes^{} c \mapsto c\,v$. 
\sk

The monoidal category ${}^H\MMM$ is also braided.
The braiding is the natural isomorphism $\tau^{}: \otimes^{} \Rightarrow \otimes^\op$
with components given in terms of the universal $R$-matrix $R = R^{(1)}\otimes R^{(2)}\in H\otimes H$ of $H$ 
 by
\begin{flalign}
\tau: V \otimes^{} W \longrightarrow W \otimes^{} V ~, \qquad v \otimes^{} w 
\longmapsto (-1)^{\vert v \vert\, \vert w \vert} \, \big(R^{(2)} \ra w\big) \otimes^{} \big(R^{(1)} \ra v\big)~,
\end{flalign}
for all homogeneous $v \in V $ and $w \in W$.
\sk

Finally, ${}^H\MMM$ is a braided closed monoidal category with
internal hom-functor
\begin{flalign}
\hom^{}: \big({}^H_{}\MMM^{}\big)^\op \times {}^H_{}\MMM^{} \longrightarrow {}^H_{}\MMM^{}~.
\end{flalign}
For any object $(V,W)$ in $({}^H_{}\MMM^{})^\op\times {}^H\MMM$
the internal hom-object $\hom(V,W) := (\udl{\hom}(\udl{V},\udl{W}),\ra)$ 
in ${}^H_{}\MMM^{}$ is given by the $\bbZ$-graded $k$-module
(\ref{eqn:internalhomgraded}) equipped with the left adjoint $H$-action
\begin{flalign}
\nn \ra: H \otimes^{} \udl{\hom}^{}\big(\udl{V}, \udl{W}\big) &\longrightarrow
 \udl{\hom}^{}\big(\udl{V}, \udl{W}\big)~, \\
h \otimes L &\longmapsto \big(h_{(1)} \ra \,\cdot\,\big) \circ \,L \,\circ\, \big(S(h_{(2)}) \ra \,\cdot\,\big)~,
\end{flalign}
where $S$ is part of the quasi-antipode $(S,\alpha,\beta)$ of $H$.
To any $ ({}^H_{}\MMM^{})^\op \times {}^H_{}\MMM^{} $-morphism 
$ (f^\op:V \to V^\prime, g:W \to W^\prime\,) $ the internal hom-functor
assigns the $ {}^H_{}\MMM^{} $-morphism induced by (\ref{eqn:internalhommorph}).
The natural currying isomorphism
\begin{flalign}
\zeta: \Hom^{}_{{}^H\MMM}(\text{--}\otimes\text{--},\text{--}) \Longrightarrow 
\Hom_{{}^H\MMM}^{}(\text{--},\hom^{}(\text{--},\text{--}))
\end{flalign} 
has components given by
\begin{flalign}\label{eqn:curryingmap}
\nn \zeta^{}(f) : \bol{V} &\longrightarrow \hom^{}(\bol{W},\bol{X})~,\\
v&\longmapsto 
f\Big(\big(\phi^{(-1)}\ra v\big)\otimes^{} \big(\big(\phi^{(-2)}\,\beta\,S(\phi^{(-3)})\big)\ra(\,\cdot\,)\big)\Big)~,
\end{flalign}
for all ${}^H_{}\MMM^{}$-morphisms $f : \bol{V}\otimes \bol{W} \to \bol{X}$, where
$\phi^{-1} = \phi^{(-1)}\otimes \phi^{(-2)}\otimes \phi^{(-3)}\in H\otimes H\otimes H$  denotes
the inverse associator of $H$.
The natural inverse currying isomorphism has components given by
\begin{flalign}\label{eqn:inversecurrying}
\nn \zeta^{-1}(g) : \bol{V}\otimes^{} \bol{W} & \longrightarrow \bol{X}~,\\
 v\otimes^{} w  & \longmapsto \phi^{(1)}\ra \Big(g(v)\big(\big(S(\phi^{(2)})\,\alpha\,\phi^{(3)}\big)\ra w\big)\Big)~,
\end{flalign}
for all ${}^H_{}\MMM^{}$-morphisms $g : \bol{V}\to\hom(\bol{W},\bol{X})$.

%%%%%%%%%%%%%%%%%%%%%%%%%%%%%%%%%%%%%%%%%%%%%%%%%%%%%%%
%%%%%%%%%%%%%%%%%%%%%%%%%%%%%%%%%%%%%%%%%%%%%%%%%%%%%%%

\subsection{Internal evaluation, composition and tensor product}
In any closed monoidal category there are evaluation and composition morphisms
for its internal homomorphisms. If the category is in addition braided then 
there are also tensor product morphisms for its internal homomorphisms. These morphisms
are derived using the currying bijection, see e.g.\ \cite[Proposition~9.3.13]{Majidbook} 
and~\cite{Barnes:2014}, and they induce important 
structures on the internal homomorphisms that make them look like morphisms.
\sk

We shall denote the internal evaluation ${}^H\MMM$-morphisms
by
\begin{flalign}\label{eqn:evaluationgeneral}
\ev : \hom^{}\big(V,W\big) \otimes^{} V &\longrightarrow W~,
\end{flalign}
the internal composition  ${}^H\MMM$-morphisms by
\begin{flalign}\label{eqn:compositiongeneral}
\bullet : \hom^{}\big(W,X\big) \otimes^{} \hom^{}\big(V,W\big) \longrightarrow \hom^{}\big(V,X\big)~,
\end{flalign}
and the internal tensor product ${}^H\MMM$-morphisms by
\begin{flalign}\label{eqn:tensorprod}
\obultimes :
\hom^{}\big(V,W\big) \otimes^{} \hom^{}\big(X,Y\big) \longrightarrow \hom^{}\big(V \otimes^{} X,W \otimes^{} Y\big)~.
\end{flalign}
For the explicit forms of these morphisms in terms of the currying map see e.g.\ 
\cite[Propositions 2.11 and 5.5]{Barnes:2014}.
The next three lemmas summarize important properties of these morphisms, which we shall frequently 
use throughout this paper.
\begin{lem}\label{lem:compproperties}
Let $V,W,X,Y$ be any four objects in ${}^H\MMM$. Then
\begin{subequations}
\begin{flalign}\label{eqn:evcompcompatibility}
\ev \big(\big(L^\prime\bullet L \big)\otimes v\big) &= \ev\Big(\big(\phi^{(1)}\ra L^\prime\, \big)\otimes \ev\Big(\big(\phi^{(2)}\ra 
L\big) \otimes \big(\phi^{(3)}\ra v\big)\Big)\Big)~, \\[4pt] \label{eqn:compweakassociativity}
\big(L^{\prime\prime}\bullet L^\prime\, \big) \bullet L
 &= (\phi^{(1)}\ra L^{\prime\prime}\, )\bullet
 \big((\phi^{(2)}\ra L^\prime\, ) \bullet
 (\phi^{(3)}\ra L)\big)~,
\end{flalign}
\end{subequations}
for all $L\in\hom(\bol{V},\bol{W})$, 
$L^\prime\in \hom(\bol{W},\bol{X})$, $L^{\prime\prime}\in\hom(\bol{X},\bol{Y})$ and $v\in V$.
\end{lem}
\begin{proof}
The proof follows the same steps as in~\cite[Proposition~2.12]{Barnes:2014}.
\end{proof}
\begin{lem}\label{lem:tensorproductproperties}
Let $V,W,X,Y$ be any four objects in ${}^H\MMM$. Then
\begin{subequations}\label{eqn:compobultimesprop}
\begin{flalign}
\label{eqn:obultimesintermsofcirc}
L\obultimes L^\prime &= \big(L\obultimes 1_{Y} \big)
\bullet \big(1_V \obultimes L^\prime\, \big)~, \\[4pt]
\label{eqn:compobultimesprop1}
 \big(K\bullet L\big)\obultimes 1_Y &= \big(K\obultimes 1_Y \big)\bullet \big(L\obultimes 1_Y \big) ~,\\[4pt]
\label{eqn:compobultimesprop2}
 1_Y \obultimes \big(K\bullet L\, \big) &= \big(1_Y \obultimes K\, \big)\bullet \big(1_Y \obultimes L\, \big)~,\\[4pt]
\label{eqn:compobultimesprop3}
(-1)^{\vert L \vert\, \vert L^\prime \vert }\, \big(R^{(2)}\ra L\big)\obultimes \big(R^{(1)}\ra L^\prime\, \big) &= \big( 1_W \obultimes L^\prime\, \big) \bullet \big(L\obultimes 1_X \big)~,
\end{flalign}
\end{subequations}
for all $L\in \hom(V,W)$, $K \in \hom(W,X)$ 
and $L^\prime \in \hom(X,Y)$, where $1_V := (\beta\ra \,\cdot\,) \in \hom(V,V)$, for all objects $V$ in ${}^H\MMM$,
are the unit internal homomorphisms.
\end{lem}
\begin{proof}
The proof follows the same steps as in~\cite[Lemmas 5.6 and 5.7]{Barnes:2014}.
\end{proof}
\begin{lem}\label{lem:obultimesphi}
Let $U,V,W,X,Y,Z$ be any six objects in ${}^H\MMM$. Then
\begin{subequations}
\begin{flalign}
\Phi\circ \big( (L\obultimes 1_W)\obultimes 1_Y\big)\circ \Phi^{-1} &=
L\obultimes\big(1_W\obultimes 1_Y\big)~,\\[4pt] 
\Phi\circ \big( (1_U\obultimes L^\prime\, )\obultimes 1_Y\big)\circ
\Phi^{-1} &= 1_U\obultimes \big(L^\prime\obultimes 1_Y\big)~,\\[4pt]
\Phi\circ \big( (1_U\obultimes 1_W)\obultimes L^{\prime\prime}\,
\big)\circ \Phi^{-1} &= 1_U\obultimes\big(1_W\obultimes L^{\prime\prime}\,
\big)~,
\end{flalign}
\end{subequations}
for all $L\in \hom(U,V)$, $L^\prime \in \hom(W,X)$ 
and $L^{\prime\prime} \in \hom(Y,Z)$.
\end{lem}
\begin{proof}
The proof follows the same steps as in~\cite[Proposition~5.9]{Barnes:2014}.
\end{proof}
\begin{rem}
To simplify notation, in what follows we shall drop  
the labels on the unit internal homomorphisms and simply write
$1 := (\beta\ra\,\cdot\,)\in\hom(V,V)$, for any object $V$ in ${}^H\MMM$.
\end{rem}

Finally, we prove a technical lemma that will be useful 
for our analysis throughout this paper.
\begin{lem}\label{lem:evcurry}
Let $V,W,X$ be any three objects in ${}^H_{}\MMM$.
\begin{itemize}
\item[(i)]  For any ${}^H_{}\MMM$-morphism $g: V\to \hom(W,X)$ there is an identity
\begin{flalign}
\zeta^{-1}(g) = \ev \circ \big(g\otimes \id \big) : V\otimes W\longrightarrow X~.
\end{flalign}
\item[(ii)] Let $f: V\otimes W \to X$ be any ${}^H_{}\MMM$-morphism.
Then $\zeta(f) \circ j =0$ if and only if $f \circ (j\otimes \id) =0$, for all ${}^H_{}\MMM$-morphisms $j : U\to V$.
\end{itemize}
\end{lem}
\begin{proof}
The proof of item (i) is exactly as in~\cite[Proposition~2.12 (i)]{Barnes:2014}.
To prove item (ii), let us first suppose that $\zeta(f) \circ j =0$. Then
\begin{flalign}
0 = \ev \circ \big(\big(\zeta(f) \circ j \big)\otimes \id\big) = \ev\circ \big(\zeta(f) \otimes \id\big)\circ (j\otimes\id)
=f\circ (j\otimes\id)~,
\end{flalign}
where in the last equality we have used item (i).
Let us now assume that $f \circ (j\otimes \id) =0$. Then
\begin{flalign}
\nn 0 &= \zeta\big(f \circ (j\otimes \id) \big) \\[4pt] \nn &= \zeta\big(\Hom_{{}^H\MMM}^{}(j^\op\otimes \id^\op,\id)\big(f\big)\big)\\[4pt]
&= \Hom_{{}^H\MMM}^{}(j^\op,\hom(\id^\op,\id)) \big(\zeta(f)\big) = \zeta(f)\circ j~,
\end{flalign}
where in the third equality we have used naturality of the currying bijection, see also 
the proof of~\cite[Theorem~2.10]{Barnes:2014}.
\end{proof}

%%%%%%%%%%%%%%%%%%%%%%%%%%%%%%%%%%%%%%%%%%%%%%%%%%%%%%%
%%%%%%%%%%%%%%%%%%%%%%%%%%%%%%%%%%%%%%%%%%%%%%%%%%%%%%%

\subsection{Algebras and bimodules}
An algebra in the braided closed monoidal category 
${}^H\MMM$ is an object $A$ in ${}^H\MMM$ 
together with two ${}^H\MMM$-morphisms 
$\mu : A\otimes A\to A$ (product) and $\eta : I \to A$ (unit), such that (denoting the product by juxtaposition 
and the unit element by $1:= \eta(1)\in A$)
\begin{flalign}
(a\,a^\prime\, )\,a^{\prime\prime} = (\phi^{(1)}\ra
a)\,\big((\phi^{(2)} \ra a^\prime\,
)\,(\phi^{(3)}\ra a^{\prime\prime}\, )\big)~,
\end{flalign}
for all $a,a^\prime,a^{\prime\prime}\in\bol{A}$, and
\begin{flalign}
1 \,a = a = a\,1 ~, 
\end{flalign}
for all $a\in\bol{A}$. To simplify notation, we shall denote an algebra in ${}^H\MMM$ simply by
its underlying bounded $\bbZ$-graded left $H$-module $A$, suppressing the product and unit from the notation. 
We denote by ${}^H\AAA$ the category of all algebras in ${}^H\MMM$; the morphisms in ${}^H\AAA$
are given by all ${}^H\MMM$-morphisms $f: A\to B$ that preserve the products and units, i.e.\
$f(a\,a^\prime\,) = f(a)\,f(a^\prime\,)$, for all $a,a^\prime\in A$, and $f(1) =1$.
\begin{ex}\label{defi:internalendalgebra}
Let $V$ be any object in $^H{}\MMM$. Then the internal endomorphism object 
$\mathrm{end}(V) := \hom(V,V)$, together with the 
$^H{}\MMM$-morphisms $\bullet : \mathrm{end}(V)\otimes \mathrm{end}(V) 
\rightarrow \mathrm{end}(V)$ and $\eta : I \rightarrow \mathrm{end}(V)\,,~c\mapsto c\,(\beta\ra\,\cdot\,)$,
is an object in ${}^H\AAA$. We call this object the  algebra of internal endomorphisms of $V$.
\end{ex}

We say that an object $A$ in ${}^H\AAA$ is braided commutative if its product is compatible
with the braiding in ${}^H\MMM$, i.e.\ $\mu\circ \tau = \mu$ or
\begin{flalign}
a \, a^\prime = (-1)^{\vert a \vert \,\vert a^\prime \vert}\, \big(R^{(2)} \ra a^\prime\, \big) \, \big(R^{(1)} \ra a \big)~,
\end{flalign}
for all homogeneous $a, a^\prime \in \bol{A}$. The full subcategory
of braided commutative algebras in ${}^H\MMM$ is denoted by ${}^H\AAA^{\mathrm{com}}$.
\begin{ex}\label{ex:classicalGMan}
Classical examples of braided commutative algebras are given by function algebras $C^\infty(M)$
on $G$-manifolds $M$, where $G$ is a Lie group with Lie algebra $\mathfrak{g}$, see~\cite[Section~6]{Barnes:2014}.
The relevant quasitriangular quasi-Hopf algebra in this case is the universal enveloping Hopf algebra
$U\mathfrak{g}$ (with trivial $R$-matrix and associator). These examples are braided commutative algebras concentrated in $\bbZ$-degree $0$.
Similarly to~\cite[Section~6]{Barnes:2014}, one can show that the exterior algebras of differential forms
$\Omega^\sharp(M)$ on $G$-manifolds are braided commutative algebras according to our definition.
These algebras are now nontrivially $\bbZ$-graded and they satisfy our assumption of boundedness of
$\bbZ$-graded $k$-modules for all finite-dimensional manifolds. These examples and cochain twist
deformations thereof (cf.~\cite[Section~6]{Barnes:2014}) are our main examples of interest.
\end{ex}

Let $A$ be any object in ${}^H\AAA$. An $A$-bimodule in ${}^H\MMM$ 
is an object $V$ in ${}^H\MMM$ together with two ${}^H\MMM$-morphisms 
$l : A\otimes V\to V$  (left $A$-action) and $r : V\otimes A \to V$  (right $A$-action),
such that (denoting also the $A$-actions by juxtaposition)
\begin{subequations}
\begin{flalign}
(v\,a)\,a^\prime &= (\phi^{(1)}\ra
v)\,\big((\phi^{(2)}\ra a)\, (\phi^{(3)}\ra
a^\prime\, )\big)~,\\[4pt]
a\,(a^\prime\, v) &= \big((\phi^{(-1)}\ra a
)\,(\phi^{(-2)}\ra a^\prime\, )\big)\,(\phi^{(-3)}\ra v)~,\\[4pt]
a\,(v\,a^\prime\, ) &=\big((\phi^{(-1)}\ra a
)\,(\phi^{(-2)}\ra v)\big)\,(\phi^{(-3)}\ra
a^\prime\, )~,
\end{flalign}
\end{subequations}
for all $a,a^\prime\in\bol{A}$ and $v\in\bol{V}$, and
\begin{flalign}
1 \,v = v = v\,1~,
\end{flalign}
for all $v\in\bol{V}$. To simplify notation, we shall denote an $A$-bimodule in ${}^H\MMM$ simply by
its underlying bounded $\bbZ$-graded left $H$-module $V$, suppressing the $A$-actions from the notation. 
We denote by ${}^H_{}{}^{}_{\bol{A}}\MMM^{}_{\bol{A}}$ the category of 
all $A$-bimodules in ${}^H\MMM$; the morphisms in ${}^H_{}{}^{}_{\bol{A}}\MMM^{}_{\bol{A}}$
are given by all ${}^H\MMM$-morphisms $f: V\to W$ that preserve the $A$-actions, i.e.\
$f(v\,a) = f(v)\,a$ and $f(a\,v) = a\,f(v)$, for all $v\in V$ and $a\in A$.
If $A$ is an object in ${}^H\AAA^{\mathrm{com}}$, we may demand that
the left and right $A$-actions in an $A$-bimodule $V$ are compatible with the braiding in
${}^H\MMM$, i.e.\ $r\circ \tau = l$ and $l\circ \tau =r$, or
\begin{subequations}\label{eqn:wcbimod}
\begin{flalign}
\label{eqn:wcbimodl}a\,v &= (-1)^{\vert a \vert\, \vert v \vert}\, \big(R^{(2)}\ra v\big) \, \big(R^{(1)}\ra a\big)~,\\[4pt]
\label{eqn:wcbimodr} v\,a &= (-1)^{\vert a \vert\, \vert v \vert}\, \big(R^{(2)}\ra a\big)\,\big(R^{(1)}\ra v\big)~,
\end{flalign}
\end{subequations}
for all homogeneous $a\in A$ and $v\in V$. We shall call such $A$-bimodules symmetric
and denote the full subcategory of all symmetric $A$-bimodules in ${}^H\MMM$
by ${}^H_{}{}^{}_{\bol{A}}\MMM^{\mathrm{sym}}_{\bol{A}}$.
\begin{ex}\label{ex:Gequibundles}
Following~\cite[Section~6]{Barnes:2014} and Example \ref{ex:classicalGMan}, the $C^\infty(M)$-bimodules 
of sections $\Gamma^\infty(E)$ of $G$-equivariant vector bundles $E\to M$ 
over $G$-manifolds $M$ are symmetric $C^\infty(M)$-bimodules concentrated in $\bbZ$-degree $0$.
Similarly, the $\Omega^\sharp(M)$-bimodules of $E$-valued differential forms $\Omega^\sharp(M,E)$
are symmetric $\Omega^\sharp(M)$-bimodules that are now nontrivially $\bbZ$-graded.
These examples and cochain twist
deformations thereof (cf.~\cite[Section~6]{Barnes:2014}) are our main examples of interest.
\end{ex}

The category ${}^H_{}{}^{}_{\bol{A}}\MMM^{\mathrm{sym}}_{\bol{A}}$ is
a braided monoidal category. The monoidal functor is denoted
$\otimes_A^{} : {}^H_{}{}^{}_{\bol{A}}\MMM^{\mathrm{sym}}_{\bol{A}}\times 
{}^H_{}{}^{}_{\bol{A}}\MMM^{\mathrm{sym}}_{\bol{A}}\to {}^H_{}{}^{}_{\bol{A}}\MMM^{\mathrm{sym}}_{\bol{A}}$ and it
assigns to any two objects $V,W$ in ${}^H_{}{}^{}_{\bol{A}}\MMM^{\mathrm{sym}}_{\bol{A}}$ the object
\begin{flalign}\label{eqn:otimesA}
\bol{V}\otimes_{\bol{A}}^{}\bol{W} = \frac{\bol{V}\otimes\bol{W}}{\mathrm{Im}\big(r \otimes \id -
(\id \otimes l)\circ \Phi\big)}
\end{flalign}
in ${}^H\MMM$, together with left and right $A$-actions given by the ${}^H\MMM$-morphisms
\begin{subequations}
\begin{flalign}
\nn l : \bol{A}\otimes\big(\bol{V}\otimes_A \bol{W}\big)
&\longrightarrow\bol{V}\otimes_A \bol{W}~, \\
 a \otimes \big(v\otimes_A w\big) &\longmapsto a\,(v\otimes_A w) := 
 \big((\phi^{(-1)}\ra a)\,(\phi^{(-2)}\ra v)\big)\otimes_A (\phi^{(-3)}\ra w)~,
\end{flalign}
and
\begin{flalign}
\nn r : \big(\bol{V}\otimes_A \bol{W}\big)\otimes\bol{A}
&\longrightarrow \bol{V}\otimes_A \bol{W}~, \\
 \big(v\otimes_A w \big)\otimes a &\longmapsto  (v\otimes_A w)\,a := 
 (\phi^{(1)}\ra v)\otimes_A \big((\phi^{(2)}\ra w) \,(\phi^{(3)}\ra a)\big)~.
\end{flalign}
\end{subequations}
To any ${}^H_{}{}^{}_{\bol{A}}\MMM^{\mathrm{sym}}_{\bol{A}}\times 
{}^H_{}{}^{}_{\bol{A}}\MMM^{\mathrm{sym}}_{\bol{A}}$-morphism 
$\big(f:\bol{V}\to\bol{X} ,g:\bol{W}\to\bol{Y}\big)$ the monoidal functor assigns
the ${}^H_{}{}^{}_{\bol{A}}\MMM^{\mathrm{sym}}_{\bol{A}}$-morphism 
\begin{flalign}
f\otimes_{\bol{A}}^{} g : \bol{V}\otimes_{\bol{A}}^{}\bol{W} \longrightarrow  \bol{X} \otimes_{\bol{A}}^{}\bol{Y}
~, \qquad v\otimes_A^{}w \longmapsto f(v)\otimes_{\bol{A}}^{} g(w)~.
\end{flalign}
The unit object in ${}^H_{}{}^{}_{\bol{A}}\MMM^{\mathrm{sym}}_{\bol{A}}$ is $A$ itself
with left and right $A$-actions given by the product in $A$. 
The unitors  in ${}^H_{}{}^{}_{\bol{A}}\MMM^{\mathrm{sym}}_{\bol{A}}$ are the natural isomorphisms
$\rho^A: \text{--} \otimes_A A \Rightarrow \id_{{}^H_{}{}^{}_{\bol{A}}\MMM^{\mathrm{sym}}_{\bol{A}}}$ 
and $\lambda^A: A \otimes_A \text{--} \Rightarrow \id_{{}^H_{}{}^{}_{\bol{A}}\MMM^{\mathrm{sym}}_{\bol{A}}}$
with components given by $\lambda^{\bol{A}}: A \otimes_A V \rightarrow V\, ,~ 
a\otimes_A v \mapsto a\,v $ and $\rho^{\bol{A}}: V \otimes_A A \rightarrow V\,,~ v\otimes_{\bol{A}}^{} a \mapsto v\,a$.
The associator is the natural isomorphism 
$\Phi^A: \otimes_A \circ \big(\otimes_A \times \id_{{}^H_{}{}^{}_{\bol{A}}\MMM^{\mathrm{sym}}_{\bol{A}}}\big) 
\Rightarrow \otimes_A\circ \big(\id_{{}^H_{}{}^{}_{\bol{A}}\MMM^{\mathrm{sym}}_{\bol{A}}}\times \otimes_A \big)$
whose components are given by
\begin{flalign}
 \nn \Phi^A:  (V\otimes_A^{} W) \otimes_A^{} X &\longrightarrow V\otimes_A^{} (W\otimes_A^{} X)~,\\
(v \otimes_A w) \otimes_A x &\longmapsto  (\phi^{(1)}\ra v) 
 \otimes_A \big((\phi^{(2)}\ra w)\otimes_A (\phi^{(3)}\ra x)\big)~.
\end{flalign}
Finally, the braiding in ${}^H_{}{}^{}_{\bol{A}}\MMM^{\mathrm{sym}}_{\bol{A}}$ is the
natural isomorphism
$\tau^A : \otimes_A \Rightarrow \otimes_A^{\op}$ 
with components given by
\begin{flalign}
\tau^A: V \otimes_A W \longrightarrow W \otimes_A V ~, \qquad v 
\otimes_A w \longmapsto (-1)^{\vert v \vert \, \vert w \vert}\, \big(R^{(2)} \ra w\big) \otimes_A \big(R^{(1)} \ra v\big)~,
\end{flalign}
for all homogeneous  $v \in V$ and $w \in W$.
The braided monoidal category ${}^H_{}{}^{}_{\bol{A}}\MMM^{\mathrm{sym}}_{\bol{A}}$ is also closed; we shall give an explicit description of the internal hom-functor
in ${}^H_{}{}^{}_{\bol{A}}\MMM^{\mathrm{sym}}_{\bol{A}}$ in Section~\ref{subsec:internalhomHAMA}.

%%%%%%%%%%%%%%%%%%%%%%%%%%%%%%%%%%%%%%%%%%%%%%%%%%%%%%%
%%%%%%%%%%%%%%%%%%%%%%%%%%%%%%%%%%%%%%%%%%%%%%%%%%%%%%%

\section{\label{sec:derivations}Derivations and differential operators}
In the remainder of this paper we shall systematically build up notions of 
differential geometry internal to the bounded $\bbZ$-graded representation category $^H\MMM$ 
of a quasitriangular quasi-Hopf algebra $H$. In this section we shall address the notions of derivations, 
differential operators and differential calculi. We describe derivations and differential 
operators as subobjects of the internal endomorphisms in $^H\MMM$ by expressing the
algebraic properties which characterize them in terms of universal categorical constructions.
See~\cite{MacLane:1998} for an introduction to the notions of limits and colimits in a category that we use below.

%%%%%%%%%%%%%%%%%%%%%%%%%%%%%%%%%%%%%%%%%%%%%%%%%%%%%%%
%%%%%%%%%%%%%%%%%%%%%%%%%%%%%%%%%%%%%%%%%%%%%%%%%%%%%%%

\subsection{Internal commutators}
Recalling Example
 \ref{defi:internalendalgebra}, for any object $V$ in ${}^H\MMM$ there exists an algebra
 in ${}^H\MMM$ given by the internal endomorphisms $\mathrm{end}(V)$ with product
 the internal composition $\bullet$ and unit element $1 :=(\beta\ra\,\cdot\,) $.
\begin{defi}\label{def:commutator}
The internal commutator in the algebra of internal endomorphisms $\mathrm{end}(V)$
is the ${}^H\MMM$-morphism
\begin{flalign}\label{eqn:commutator}
[\, \cdot\, ,\, \cdot\, ]: \mathrm{end}(V) \otimes \mathrm{end}(V) \longrightarrow \mathrm{end}(V)~, \qquad
L\otimes L^\prime \longmapsto (\bullet - \bullet \circ \tau)(L \otimes L^\prime\, )~.
\end{flalign}
\end{defi}

\begin{propo}\label{prop:bracket}
The internal commutator in $\mathrm{end}(V)$ satisfies the following properties:
\begin{itemize}
\item[(i)] If $H$ is triangular, i.e.\ its $R$-matrix satisfies $R = R_{21}^{-1}$, then 
$ [\,\cdot\, ,\, \cdot\,] $ is braided antisymmetric, i.e.\
\begin{flalign}
[\, \cdot\, , \,\cdot\, ] = - [\,\cdot\,, \,\cdot\,] \circ \tau~,
\end{flalign}
or
\begin{flalign}
\big[L,L^\prime\, \big] = - (-1)^{\vert L\vert \,\vert L^\prime\vert}\, \big[R^{(2)}\ra L^\prime , R^{(1)}\ra L\big]~,
\end{flalign}
for all homogeneous $L,L^\prime\in \mathrm{end}(V)$.

\item[(ii)] If $H$ is triangular, then $ [\,\cdot\, ,\, \cdot\,] $ 
satisfies the braided Jacobi identity $ \mathrm{Jac} = 0 $, with Jacobiator given by the $ ^H\MMM $-morphism 
 $\mathrm{Jac}: (\mathrm{end}(V) \otimes \mathrm{end}(V)) \otimes \mathrm{end}(V) \longrightarrow \mathrm{end}(V)$
 defined as
\begin{flalign}\label{eqn:Jacobi}
\mathrm{Jac}:= [\,\cdot\, , \,\cdot\,] \circ \big([\,\cdot\,, \,\cdot\,] \otimes \id\big) 
\circ \big(((\id \otimes \id) \otimes \id) + (\tau \circ \Phi) + (\Phi^{-1} \circ \tau) \big)~,
\end{flalign}
or
\begin{flalign}
\nn 0=&\,\big[\big[L,L^\prime\, \big],L^{\prime\prime}\,\big] \\
\nn &+ (-1)^{\vert L\vert \,(\vert L^\prime\vert + \vert L^{\prime\prime}\vert)}\, 
\big[\big[R^{(2)}_{(1)}\, \phi^{(2)}\ra L^\prime , R^{(2)}_{(2)}\, \phi^{(3)}\ra L^{\prime\prime}\, \big], R^{(1)}\, \phi^{(1)}\ra L\big] 
\\ &+ (-1)^{\vert L^{\prime\prime}\vert \, (\vert L\vert + \vert L^\prime\vert)}\, 
 \big[\big[\phi^{(-1)}\, R^{(2)}\ra L^{\prime\prime} , \phi^{(-2)}\, R^{(1)}_{(1)}\ra L \big], 
\phi^{(-3)}\, R^{(1)}_{(2)}\ra L^{\prime}\, \big] ~,
\end{flalign}
for all homogeneous $L,L^\prime,L^{\prime\prime}\in \mathrm{end}(V)$.

\item[(iii)] For generic quasitriangular $H$,
$ [\, \cdot\, , \,\cdot\, ] $ satisfies the braided derivation property
\begin{flalign}\label{eqn:bideriv}
[\, \cdot\, ,\, \cdot\, ] \circ (\bullet \otimes \id) = 
\bullet \circ \Big(\big(\id \otimes [\, \cdot\, ,\, \cdot\, ]\big) + \big([\, \cdot\, ,\, \cdot\, ] \otimes \id\big) 
\circ \Phi^{-1} \circ \big(\id \otimes \tau\big) \Big)\circ \Phi~,
\end{flalign}
or
\begin{multline}
\big[L \bullet L^{\prime}, L^{\prime\prime}\, \big] = \big(\phi^{(1)} \ra L\big)\bullet\big[\phi^{(2)}\ra L^{\prime}, \phi^{(3)} \ra L^{\prime\prime}\, \big] \\
+ (-1)^{\vert L^{\prime} \vert \,\vert L^{\prime\prime} \vert}~
\big[\, \widetilde{\phi}^{(-1)}\, \phi^{(1)} \ra L , \widetilde{\phi}^{(-2)}\, R^{(2)} \, 
\phi^{(3)} \ra L^{\prime\prime}\big] \bullet \big(\widetilde{\phi}^{(-3)} \, R^{(1)} \, \phi^{(2)} \ra L^{\prime}\, \big)~,
\end{multline}
for all homogeneous $ L,L^{\prime},L^{\prime\prime} \in \mathrm{end}(V) $.
\end{itemize}
\end{propo}
\begin{proof}
Item (i) follows from a short calculation 
\begin{flalign}
[\, \cdot\, , \, \cdot\, ]  = \bullet - \bullet \circ \tau =
 -(\bullet \circ \tau - \bullet) = 
 -(\bullet- \bullet \circ \tau^{-1} ) \circ \tau = 
  -(\bullet - \bullet \circ \tau ) \circ \tau = - [\, \cdot\, , \, \cdot\, ] \circ \tau~,
\end{flalign}
where in the fourth equality we have used triangularity of the $R$-matrix which implies $\tau^{-1}=\tau$. 
The proofs of items (ii) and (iii) involve standard manipulations using the weak 
associativity of the internal composition \eqref{eqn:compweakassociativity}
and standard properties of the $R$-matrix (see e.g.\ \cite[Section~5.1]{Barnes:2014}).
\end{proof}
\begin{cor}\label{cor:endisLie}
Let $H$ be a triangular quasi-Hopf algebra and  $V$ any object in ${}^H\MMM$. 
Then the ${}^H\MMM$-object given by the internal endomorphisms $\mathrm{end}(V)$, together with
the internal commutator $[\,\cdot\,,\,\cdot\,]$ given in (\ref{eqn:commutator}),
is a Lie algebra in ${}^H\MMM$.
\end{cor}

%%%%%%%%%%%%%%%%%%%%%%%%%%%%%%%%%%%%%%%%%%%%%%%%%%%%%%%
%%%%%%%%%%%%%%%%%%%%%%%%%%%%%%%%%%%%%%%%%%%%%%%%%%%%%%%

\subsection{Derivations}
We give a description of the derivations on an object $A$ in $^H\AAA^{\mathrm{com}}$
by using universal constructions in the braided closed monoidal category ${}^H\MMM$ to formalize a suitable version of the Leibniz rule, that is 
compatible with the structures in ${}^H\MMM$, in terms of an equalizer.
Let us start by noticing 
that for any object $V$ in ${}^H_{}{}^{}_{\bol{A}}\MMM{}^{\mathrm{sym}}_{\bol{A}}$ 
there is an ${}^H\MMM$-morphism
\begin{flalign}\label{eqn:leftAactionvsrepresentation1}
\widehat{l} := \zeta(l) : \bol{A}\longrightarrow \mathrm{end}(\bol{V})~,
\end{flalign}
which is obtained by currying the left $A$-action $l : A\otimes V\to V$.
Similarly to~\cite[Lemma~4.1]{Barnes:2014}, one can show that (\ref{eqn:leftAactionvsrepresentation1}) is
moreover an ${}^H\AAA$-morphism to the algebra of internal endomorphisms, cf.\ Example \ref{defi:internalendalgebra}.
In particular, this implies that for any object $A$ in ${}^H\AAA^{\mathrm{com}}$ 
there exist two parallel $^H\MMM$-morphisms
\begin{flalign}
\xymatrix{
\mathrm{end}(A) \otimes A  \ar@<-1ex>[rrr]_-{\widehat{l} \circ \ev} \ar@<1ex>[rrr]^-{[\, \cdot\, , \, \cdot\, ]} 
 &&& \mathrm{end}(A)
}~,
\end{flalign}
where for brevity we denote by $[\, \cdot\, , \, \cdot\, ]$ the 
composition $[\, \cdot\, , \, \cdot\, ] \circ \big(\id \otimes \widehat{l} \ \big)$ with $l : A\otimes A\to A$
the left $A$-action induced by the product in $A$. 
\begin{defi}\label{defi:derivations}
Let $A$ be an object in $^H\AAA^{\mathrm{com}}$. The derivations of $A$ is the object $\mathrm{der}(A)$ 
in ${}^H\MMM$ which is defined by the equalizer
\begin{flalign}\label{eqn:derivations}
\xymatrix{
\mathrm{der}(A)\ar[rr]&& \mathrm{end}(A) \ar@<-1ex>[rrr]_-{\zeta(\widehat{l} \circ \ev)} \ar@<1ex>[rrr]^-{\zeta([\, \cdot\,,\, \cdot\,  ])}  &&& \hom(A, \mathrm{end}(A))
}
\end{flalign}
in ${}^H\MMM$.
 \end{defi}

In the category $^H\MMM$ equalizers may be computed by taking the kernel of the difference
of the two parallel morphisms. In particular, $\mathrm{der}(A)$ can be represented explicitly as the kernel
\begin{flalign}\label{eqn:der}
\mathrm{der}(A) = \mathrm{Ker}\Big(\zeta \big([\, \cdot\, , \,\cdot\, ] - \widehat{l} \circ \ev\big) \Big)~.
\end{flalign}

The following lemma will allow us to establish a relation between our definition of derivations
and the standard definition in terms of a Leibniz rule.
\begin{lem}\label{lem:der}
Let $A$ be any object in $^H\AAA^{\mathrm{com}}$. An $^H\MMM$-subobject $U \subseteq \mathrm{end}(A)$ 
is an $^H\MMM$-subobject of $\mathrm{der}(A)$ if and only if
\begin{flalign}\label{egn:derbracket}
[L, a]  = \widehat{l}(\ev ( L \otimes a) )~,
\end{flalign}
for all $L \in U$ and $a \in A$.
\end{lem}
\begin{proof}
Denoting by $f := [\, \cdot\, , \,\cdot\, ] - \widehat{l} \circ \ev : \mathrm{end}(A)\otimes A \to \mathrm{end}(A)$
and $j : U \to \mathrm{end}(A) $ the inclusion ${}^H\MMM$-morphism,
we have to show that $\zeta(f) \circ j =0$ if and only if $f\circ (j\otimes \id)=0$.
This is a consequence of item (ii) of Lemma~\ref{lem:evcurry}.
\end{proof}
\begin{rem}\label{rem:usualderiv}
We explain how our definition of derivations is related to the standard 
definition in terms of a Leibniz rule:
Let $L\in\mathrm{der}(A)$ be any derivation. Then Lemma~\ref{lem:der} implies that
$[L, a]  = \widehat{l}(\ev ( L \otimes a) )$. Evaluating this equation on some $a^\prime\in A$, we obtain
\begin{flalign}
\ev\big([L, a] \otimes a^\prime\,\big) =  \ev\big(\, \widehat{l}(\ev ( L \otimes a) )\otimes a^\prime\, \big)~.
\end{flalign}
Using now the evaluation identity \eqref{eqn:evcompcompatibility} and also item (i) of Lemma~\ref{lem:evcurry},
we can simplify this equation and obtain
\begin{multline}
\ev\Big((\phi^{(1)}\ra L) \otimes \big( (\phi^{(2)}\ra a)\,(\phi^{(3)}\ra a^\prime\,)\big)\Big)\\
- (-1)^{\vert L\vert\,\vert a\vert } ~
(\phi^{(1)}\, R^{(2)}\ra a)\, \ev\Big((\phi^{(2)} \, R^{(1)}\ra L)\otimes (\phi^{(3)}\ra a^\prime)\Big)
= \ev\big(L\otimes a\big)\, a^\prime~,
\end{multline}
for all homogeneous $a,a^\prime \in A$ and $L\in\mathrm{der}(A)$.
For the special case of trivial $R$-matrix $R=1\otimes 1$ and associator $\phi=1\otimes 1\otimes 1$ the last equation
reduces to $L(a \, a^\prime\,) = L(a)\,a^\prime + (-1)^{\vert L\vert\,\vert a\vert }~a\, L(a^\prime\, )$,
which is exactly the Leibniz rule for a graded derivation. Hence, the equalizer (\ref{eqn:derivations})
provides us with a suitable generalization of the graded Leibniz rule that is consistent with
the structures in the braided closed monoidal category ${}^H\MMM$.
\end{rem}

Finally, we prove a structural result for our derivations.
\begin{propo}\label{propo:der}
Let $H$ be a triangular quasi-Hopf algebra and  $A$ any object in $^H\AAA^{\mathrm{com}}$.
Then the ${}^H\MMM$-object given by the derivations $\mathrm{der}(A)$, together with 
the internal commutator $[\,\cdot\,,\,\cdot\,]$ given in (\ref{eqn:commutator}),
is a Lie algebra in ${}^H\MMM$.
\end{propo}
\begin{proof}
We already know from Corollary \ref{cor:endisLie} that, under our hypotheses,
 $\mathrm{end}(A)$ together the internal commutator $[\,\cdot\,,\,\cdot\,]$
is a Lie algebra in ${}^H\MMM$. Moreover,
$\mathrm{der}(A)$ is by construction an ${}^H\MMM$-subobject of  $\mathrm{end}(A)$, so it remains
to prove that the image of the restricted internal commutator
\begin{flalign}
[\,\cdot\,,\,\cdot\,] : \mathrm{der}(A)\otimes \mathrm{der}(A)\longrightarrow \mathrm{end}(A)
\end{flalign}
is an ${}^H\MMM$-subobject of $\mathrm{der}(A)$.
Using Lemma~\ref{lem:der} this is the case if and only if 
\begin{flalign}
[[L,L^\prime\, ], a]  = \widehat{l} \big(\ev \big( [L,L^\prime\, ] \otimes a\big) \big)~,
\end{flalign} 
for all $L,L^\prime \in \mathrm{der}(A)$ and $a \in A$. 
One can now easily show that this equality holds true by using
the braided Jacobi identity and antisymmetry (cf.\ items (ii) and (i) of Proposition~\ref{prop:bracket}), 
the derivation property of Lemma~\ref{lem:der} and finally the evaluation identity \eqref{eqn:evcompcompatibility}.
For simplifying the resulting expressions one also needs standard
$R$-matrix properties, which are listed in e.g.~\cite[Section~5.1]{Barnes:2014}.
\end{proof}

%%%%%%%%%%%%%%%%%%%%%%%%%%%%%%%%%%%%%%%%%%%%%%%%%%%%%%%
%%%%%%%%%%%%%%%%%%%%%%%%%%%%%%%%%%%%%%%%%%%%%%%%%%%%%%%

\subsection{Cochain twisting of derivations}
We shall briefly study the deformation of derivations under cochain twisting.
For a more complete introduction to these deformation techniques we refer to Part~I.
Let $H$ be a quasitriangular quasi-Hopf algebra and $F$ a cochain twisting element, i.e.\
$F\in H\otimes H$ is an invertible element with the normalization $(\epsilon\otimes \id)(F) = 1 = (\id\otimes\epsilon)(F)$.
Any cochain twisting element defines a braided closed monoidal functor
$\FF : {}^H_{}\MMM \to {}^{H_F}_{}\MMM $, where $H_F$ is the twisted
quasitriangular quasi-Hopf algebra of $H$ by $F$, see e.g.\  \cite[Theorem~5.11]{Barnes:2014}.
The functor $\FF : {}^H_{}\MMM \to {}^{H_F}_{}\MMM $ acts on objects and morphisms
as the identity, and the coherence maps for the braided monoidal structures are the 
${}^{H_F}_{}\MMM$-isomorphisms
\begin{subequations}\label{eqn:coherencemapstensor}
\begin{flalign}
\nn \varphi :  \FF(\bol{V})\otimes_F^{} \FF(\bol{W}) &\longrightarrow \FF(\bol{V}\otimes \bol{W}) ~,\\
v\otimes_F^{}  w &\longmapsto \big(F^{(-1)}\ra v \big)\otimes \big(F^{(-2)}\ra w \big)~,
\end{flalign}
where $F^{-1} = F^{(-1)}\otimes F^{(-2)}$ denotes the inverse cochain twisting element, and
\begin{flalign}\label{eqn:coherencemapstensor2}
\psi : \bol{I}_F^{} \longrightarrow \FF(\bol{I}) ~, \qquad c \longmapsto c~.
\end{flalign}
\end{subequations}
The coherence maps for the internal hom-structures are the ${}^{H_F}_{}\MMM$-isomorphisms
\begin{flalign}\label{eqn:gamma}
\nn\gamma :  \hom_F^{}\big(\FF(\bol{V}),\FF(\bol{W})\big) & \longrightarrow \FF\big(\hom(\bol{V},\bol{W})\big)~,\\
 L&\longmapsto \big( F^{(-1)}\ra \,\cdot\,\big) \circ L \circ \big( S\big(F^{(-2)}\big)\ra \,\cdot\,\big)~.
\end{flalign}
The braided closed monoidal functor $\FF : {}^H_{}\MMM \to {}^{H_F}_{}\MMM $ induces
functors (denoted with abuse of notation by the same symbols) 
$\FF : {}^H_{}\AAA^{\mathrm{com}} \to {}^{H_F}_{}\AAA^{\mathrm{com}} $ and
$\FF : {}^H_{}{}^{}_{A} \MMM_{A}^{\mathrm{sym}} \to {}^{H_F}_{}{}^{}_{\FF(A)}\MMM^{\mathrm{sym}}_{\FF(A)} $,
which allow us to twist quantize algebras and bimodules in ${}^H\MMM$ to
algebras and bimodules in ${}^{H_F}\MMM$.
Details can be found in~\cite[Proposition~5.16]{Barnes:2014}.
\begin{propo}\label{propo:cochainder}
Let $A$ be any object in  ${}^H\AAA^{\mathrm{com}}$ and let $F$ be any cochain 
twisting element based on $H$. Then the coherence map $\gamma : \mathrm{end}_F(\FF(A))\to \FF(\mathrm{end}(A))$ 
restricts to an ${}^{H_F}_{}\MMM$-isomorphism
\begin{flalign}
\gamma: \mathrm{der}_F(\FF(A)) \longrightarrow \FF(\mathrm{der}(A)) ~.
\end{flalign}
\end{propo}
\begin{proof}
The braided closed monoidal functor 
$\FF : {}^H_{}\MMM \to {}^{H_F}_{}\MMM $ is an equivalence of categories,
hence it preserves all limits and colimits. It then follows that $\FF(\mathrm{der}(A))$ is the equalizer
of the $^{H_F}\MMM$-diagram
\begin{flalign}\label{eqn:Fderdiagram}
\xymatrix{
\FF(\mathrm{end}(A))  \ar@<-1ex>[rrrr]_-{\FF(\zeta(\widehat{l} \circ \ev))} \ar@<1ex>[rrrr]^-{\FF(\zeta([\, \cdot\, ,\, \cdot\, ]))} &&&& \FF(\hom(A, \mathrm{end}(A)))
} \ .
\end{flalign}
On the other hand, the object $\mathrm{der}_F(\FF(A))$ in ${}^{H_F}\MMM$ is defined according
to Definition \ref{defi:derivations} as the equalizer of the $^{H_F}\MMM$-diagram
\begin{flalign}\label{eqn:derFdiagram}
\xymatrix{
\mathrm{end}_F(\FF(A)) \ar@<-1ex>[rrrr]_-{\zeta_F(\widehat{l}_F \circ \ev_F)} \ar@<1ex>[rrrr]^-{\zeta_F([\, \cdot\, ,\,\cdot\,]_F)}  &&&& \hom_F(\FF(A), \mathrm{end}_F(\FF(A))) \ .
}
\end{flalign}
A straightforward but slightly lengthy calculation shows that
the $^{H_F}\MMM$-diagrams (\ref{eqn:Fderdiagram}) and (\ref{eqn:derFdiagram}) 
are isomorphic: The $^{H_F}\MMM$-diagram
\begin{flalign}\label{eqn:Fpreserveslimits}
\xymatrix{
\ar[dd]_-{\gamma}
\mathrm{end}_F(\FF(A)) \ar@<-1ex>[rrrr]_-{\zeta_F(\widehat{l}_F \circ \ev_F)} \ar@<1ex>[rrrr]^-{\zeta_F([\, \cdot\, ,\,\cdot\,]_F)}  &&&& \hom_F(\FF(A), \mathrm{end}_F(\FF(A))) \ar[d]^-{\gamma\circ (\,\cdot\,)}\\
&&&& \hom_F(\FF(A), \FF(\mathrm{end}(A)))\ar[d]^-{\gamma}\\
\FF(\mathrm{end}(A))  \ar@<-1ex>[rrrr]_-{\FF(\zeta(\widehat{l} \circ \ev))} \ar@<1ex>[rrrr]^-{\FF(\zeta([\, \cdot\, ,\, \cdot\, ]))} &&&& \FF(\hom(A, \mathrm{end}(A)))
}~
\end{flalign}
commutes (i.e.\ the diagram obtained by taking either both upper or lower horizontal arrows commutes)
and the vertical arrows are all ${}^{H_F}\MMM$-isomorphisms.
Due to the universality of limits there must be a unique isomorphism between $\mathrm{der}_F(\FF(A))$ 
and $\FF(\mathrm{der}(A))$. The assertion now follows from the fact that we describe our derivations as a subobject
of the internal endomorphisms (cf.\ (\ref{eqn:der})) and hence the unique isomorphism between $\mathrm{der}_F(\FF(A))$ 
and $\FF(\mathrm{der}(A))$ is the one induced by the isomorphism between $\mathrm{end}_F(\FF(A))$ 
and $\FF(\mathrm{end}(A))$, which is precisely~$\gamma$. 
\end{proof}

%%%%%%%%%%%%%%%%%%%%%%%%%%%%%%%%%%%%%%%%%%%%%%%%%%%%%%%
%%%%%%%%%%%%%%%%%%%%%%%%%%%%%%%%%%%%%%%%%%%%%%%%%%%%%%%

\subsection{\label{subsec:internalhomHAMA}Internal homomorphisms}
In~\cite[Section~4]{Barnes:2014} we gave an explicit description of the internal hom-functor 
$\hom_A : \big({}^{H}_{}{}^{}_{\bol{A}}\MMM_{\bol{A}}^{}\big)^{\op}\times {}^{H}_{}{}^{}_{\bol{A}}\MMM_{\bol{A}}^{}\to
{}^{H}_{}{}^{}_{\bol{A}}\MMM_{\bol{A}}^{}$ by imposing a suitable weak right $A$-linearity condition
on the internal hom-functor $\hom$ in ${}^H\MMM$. We shall now give an easier but equivalent
construction of $\hom_A$ for the case where $A$ is an object in ${}^H{\AAA}^{\mathrm{com}}$
and we restrict ourselves to the full subcategory ${}^{H}_{}{}^{}_{\bol{A}}\MMM_{\bol{A}}^{\mathrm{sym}}$ 
of symmetric $A$-bimodules in ${}^H\MMM$. This construction involves a generalization of
the internal commutator $[\,\cdot\, , \,\cdot\, ]$ from Definition \ref{def:commutator}, and it will allow us later on to interpret the internal 
hom-objects $\hom_A(V,W)$ as zeroth order differential operators.
\sk

Let $A$ be an object in  $^H\AAA^{\mathrm{com}}$ and let $V,W$ be two objects
in $ {}^{H}_{}{}^{}_{\bol{A}}\MMM_{\bol{A}}^{\mathrm{sym}} $. 
We define an ${}^H\MMM$-morphism (denoted with abuse of notation by the same symbol as the internal commutator)
\begin{flalign}\label{eqn:brackethomVW}
[\,\cdot\,,\,\cdot\, ] := \bullet \circ \big( \id\otimes \widehat{l} \ \big) - \bullet \circ \big(\, \widehat{l} \otimes\id\big)
\circ \tau  : \hom(V,W) \otimes A &\longrightarrow \hom(V,W)~,
\end{flalign}
where $\widehat{l}$ was defined in (\ref{eqn:leftAactionvsrepresentation1}).
Then
\begin{flalign}
[L,a] = L\bullet \widehat{l}(a) - (-1)^{\vert L\vert\,\vert a\vert}\ \widehat{l}\big(R^{(2)}\ra a\big) \bullet 
\big(R^{(1)}\ra L\big)~,
\end{flalign} 
for all  homogeneous $L\in \hom(V,W)$ and $a\in A$.
\begin{defi}\label{defi:homAcurrent}
The object $\hom_A(V,W)$ in  $ ^H\MMM $ is defined by the equalizer
\begin{flalign}
\xymatrix{
\hom_A(V,W)\ar[rr] && \hom(V,W) \ar@<-1ex>[rrr]_-{0} \ar@<1ex>[rrr]^-{\zeta([\, \cdot\,,\, \cdot\,  ])}  &&& \hom(A, \hom(V,W))
}
\end{flalign}
in $ ^H\MMM $.
This equalizer can be realized explicitly in terms of the ${}^H\MMM$-subobject
\begin{flalign}\label{eqn:hom}
\hom_A(V,W) = \mathrm{Ker}\big(\zeta([\, \cdot\, , \, \cdot\, ])\big)\subseteq \hom(V,W)~
\end{flalign}
of the internal hom-object $\hom(V,W)$ in ${}^H\MMM$.
\end{defi}

\begin{lem}\label{lem:homA}
Let $A$ be any object in $^H\AAA^{\mathrm{com}}$ and let $V,W$
be any two objects in $ {}^{H}_{}{}^{}_{\bol{A}}\MMM_{\bol{A}}^{\mathrm{sym}} $. 
An $^H\MMM$-subobject $U \subseteq \hom(V,W)$ 
is an $^H\MMM$-subobject of $\hom_A(V,W)$ if and only if
\begin{flalign}\label{egn:hombracket}
[L, a]  = 0~,
\end{flalign}
for all $L \in U$ and $a \in A$.
\end{lem}
\begin{proof}
Denoting by $f := [\, \cdot\, , \,\cdot\, ] : \hom(V,W)\otimes A \to \hom(V,W)$
and $j : U \to \hom(V,W) $ the inclusion ${}^H\MMM$-morphism,
we have to show that $\zeta(f) \circ j =0$ if and only if $f\circ (j\otimes \id)=0$.
This is a consequence of item (ii) of Lemma~\ref{lem:evcurry}.
\end{proof}

The object $\hom_A(V,W)$ in ${}^H\MMM$ given by (\ref{eqn:hom}) 
carries a natural left and right $A$-action given by the ${}^H\MMM$-morphisms
\begin{subequations}\label{eqn:lrhomA}
\begin{flalign}
l &:= \bullet \circ (\, \widehat{l} \otimes \id): A \otimes \hom_A(V,W) \longrightarrow \hom_A(V,W)~,\\[4pt]
r &:= \bullet \circ (\id \otimes \widehat{l} \ ): \hom_A(V,W) \otimes A \longrightarrow \hom_A(V,W)~.
\end{flalign}
\end{subequations}
It is moreover an object in $ {}^{H}_{}{}^{}_{\bol{A}}\MMM_{\bol{A}}^{\mathrm{sym}} $ because the result
of Lemma~\ref{lem:homA} is precisely the symmetry condition for the left and right $A$-action
given in (\ref{eqn:lrhomA}) (see also (\ref{eqn:brackethomVW})).
The assignment of these objects $\hom_A(V,W)$ in $ {}^{H}_{}{}^{}_{\bol{A}}\MMM_{\bol{A}}^{\mathrm{sym}} $
is functorial and we denote the corresponding functor by
\begin{flalign}\label{eqn:internalhomHAMA}
\hom_A : \big( {}^{H}_{}{}^{}_{\bol{A}}\MMM_{\bol{A}}^{\mathrm{sym}} \big)^\op\times 
 {}^{H}_{}{}^{}_{\bol{A}}\MMM_{\bol{A}}^{\mathrm{sym}} \longrightarrow 
  {}^{H}_{}{}^{}_{\bol{A}}\MMM_{\bol{A}}^{\mathrm{sym}} ~.
\end{flalign}
To any $\big( {}^{H}_{}{}^{}_{\bol{A}}\MMM_{\bol{A}}^{\mathrm{sym}} \big)^\op\times 
 {}^{H}_{}{}^{}_{\bol{A}}\MMM_{\bol{A}}^{\mathrm{sym}}$-morphism
 $(f^\op : V\to V^\prime , ~ g: W\to W^\prime\, )$ this functor assigns 
\begin{flalign}
\hom_A(f^\op,g): \hom_A(V,W) \longrightarrow 
\hom_A(V^\prime,W^\prime\, )~, \qquad L \longmapsto g \circ L \circ f~.
\end{flalign} 
Finally, we show
that (\ref{eqn:internalhomHAMA}) is an internal hom-functor
in ${}^{H}_{}{}^{}_{\bol{A}}\MMM_{\bol{A}}^{\mathrm{sym}}$.
\begin{propo}
The braided monoidal category ${}^{H}_{}{}^{}_{\bol{A}}\MMM_{\bol{A}}^{\mathrm{sym}}$ is closed: There is a natural bijection $\zeta^A : \Hom_{{}^{H}_{}{}^{}_A\MMM^{\mathrm{sym}}_A}(\text{--}
\otimes_A^{}\text{--},\text{--})
\Rightarrow \Hom_{{}^{H}_{}{}^{}_A\MMM^{\mathrm{sym}}_A}(\text{--},\hom_A(\text{--},\text{--}))$ 
with components given by
\begin{flalign}
\nn \zeta^{\bol{A}} (f) : \bol{V} &\longrightarrow \hom_A(\bol{W},\bol{X})~,\\
v&\longmapsto f\Big(\big(\phi^{(-1)}\ra v \big) \otimes_{\bol{A}}^{} \big(\big(\phi^{(-2)}\,\beta\,S(\phi^{(-3)})\big) \ra (\,\cdot\,)\big)\Big)~,\label{eqn:rightcurryingA}
\end{flalign}
for all ${}^{H}_{}{}^{}_{\bol{A}}\MMM_{\bol{A}}^{\mathrm{sym}}$-morphisms $f: V\otimes_{\bol{A}}^{} W\to X$.
The components of its inverse are
\begin{flalign}
\nn  (\zeta^A)^{-1}(g) : \bol{V}\otimes_{\bol{A}}^{}\bol{W} &\longrightarrow \bol{X}~,\\
  v\otimes_{\bol{A}}^{} w & \longmapsto \phi^{(1)}\ra \Big(g(v)\big( \big(S(\phi^{(2)})\,\alpha\,\phi^{(3)} \big)\ra w\big)\Big)~,\label{eqn:inversecurryingA}
\end{flalign}
for all ${}^H_{}{}^{}_{\bol{A}}\MMM^{\mathrm{sym}}_{\bol{A}}$-morphisms $g: \bol{V}\to\hom_A(\bol{W},\bol{X})$.
\end{propo}
\begin{proof}
With a proof analogous to~\cite[Lemma~4.2]{Barnes:2014} one shows that
(\ref{eqn:rightcurryingA}) and (\ref{eqn:inversecurryingA}) are the components 
of a natural bijection between the functors
$\Hom_{{}^{H}_{}{}^{}_A\MMM}(\text{--} \otimes_A^{}\text{--},\text{--})$
and $\Hom_{{}^{H}_{}{}^{}_A\MMM{}^{}_A}(\text{--},\hom(\text{--},\text{--}))$.
It thus remains to prove that (1)\  the image of $\zeta^{\bol{A}} (f)$ is contained in $\hom_A(W,X)$ for all $ {}^{H}_{}{}^{}_{\bol{A}}\MMM^{\mathrm{sym}}_{\bol{A}} $-morphisms 
$ f: V \otimes_A W \to X $, and that
(2)\ $(\zeta^A)^{-1}(g)$ is a right $A$-linear map for all $ {}^{H}_{}{}^{}_{\bol{A}}\MMM^{\mathrm{sym}}_{\bol{A}} $-morphisms 
$g : V\to \hom_A(X,Y)$.
\sk

Due to Lemma~\ref{lem:homA}, point (1)\ is shown by the calculation
\begin{flalign}
\nn\big( \zeta^{A}(f)(v)\big) \, a &= \zeta^{A}(f)(v\, a) \\[4pt]
\nn &= (-1)^{\vert a \vert\, \vert v \vert}~\zeta^{A}(f)\big((R^{(2)} \ra a) \, (R^{(1)} \ra v)\big) \\[4pt]
&= (-1)^{\vert a \vert \,\vert v \vert}~\big(R^{(2)} \ra a \big) \,\big( R^{(1)} \ra \zeta^{A}(f)(v)\big)~,
\end{flalign}
for all homogeneous $ a \in A $ and $v \in V$. In the first equality we have used the right $A$-linearity
of $\zeta^{A}(f)$, in the second equality the symmetry of the $A$-bimodule $V$,
and in the last equality the left $A$-linearity and $H$-equivariance of $\zeta^{A}(f)$.
\sk

Point (2) is likewise shown by a short calculation
\begin{flalign}
\nn (\zeta^A)^{-1}(g)\big((v\otimes_A^{} w)\, a\big)&= 
(\zeta^A)^{-1}(g)\big((\phi^{(1)}\ra v) \otimes_A^{} \big((\phi^{(2)}\ra w) \, (\phi^{(3)}\ra a)\big)\big)\\[4pt]
\nn &= \ev\big(g (\phi^{(1)}\ra v ) \otimes_A^{} \big((\phi^{(2)}\ra w)\,(\phi^{(3)}\ra a)\big)\big)\\[4pt]
\nn &=\ev\big( (g(v)\otimes_A^{} w)\,a\big) \\[4pt]
\nn &= (-1)^{\vert a\vert\,(\vert v\vert+\vert w\vert)} ~\big( R^{(2)}\ra a \big) \,\big(R^{(1)}\ra \ev\big(g(v)\otimes_A^{} w\big)\big)\\[4pt]
&= \big((\zeta^A)^{-1}(g)(v\otimes_A^{} w)\big)\, a~,
\end{flalign}
for all homogeneous $ a \in A $, $v \in V$ and $w \in W$. 
The second equality holds by direct inspection (see also Lemma~\ref{lem:evcurry} (i) for a similar statement)
and in the fourth equality we have used the symmetry of the $A$-bimodules
$W$ and $\hom_A(V,W)$ as well as the $H$-equivariance of $\ev$.
The last equality uses the symmetry of the $A$-bimodule $X$.
\end{proof}

%%%%%%%%%%%%%%%%%%%%%%%%%%%%%%%%%%%%%%%%%%%%%%%%%%%%%%%
%%%%%%%%%%%%%%%%%%%%%%%%%%%%%%%%%%%%%%%%%%%%%%%%%%%%%%%

\subsection{\label{subsec:diffops}Differential operators and calculi}
Let $A$ be an object in ${}^H\AAA^{\mathrm{com}}$ and $V$ any object in
${}^{H}_{}{}^{}_{\bol{A}}\MMM_{\bol{A}}^{\mathrm{sym}}$. 
We define the internal multi-commutator of order $n\in\bbZ_{>0}$ 
to be the ${}^H\MMM$-morphism
\begin{subequations}\label{eqn:multicom}
\begin{flalign}
[\, \cdot\, , \, \cdot\, ]^{(n)} : 
\left(\cdots \big((\mathrm{end}(V) \otimes A) \otimes A\big) \cdots \right)\otimes A  \longrightarrow \mathrm{end}(V)~,
\end{flalign}
where the source contains $n$ factors of $A$, given by the composition
\begin{flalign}
[\, \cdot\, , \, \cdot\, ]^{(n)} := [\, \cdot\, , \, \cdot\, ] \circ \big([\, \cdot\, , \, \cdot\, ] \otimes \id \big) \circ \cdots \circ \big((\cdots (([\, \cdot\, ,\, \cdot\, ] \otimes \id) \otimes \id) \cdots )\otimes \id\big)~.
\end{flalign}
\end{subequations}
We have suppressed as before the precomposition of the internal multi-commutator 
with $(\cdots ((\id \otimes \widehat{l} \ ) \otimes \widehat{l} \ ) \cdots ) \otimes \widehat{l} \ $, where $\widehat{l}$
is the ${}^H\AAA$-morphism given in (\ref{eqn:leftAactionvsrepresentation1}).
We further denote by $\Phi^{(-n)}$ the combination of associators required to re-bracket the expressions
\begin{flalign}
\xymatrix{
 \mathrm{end}(V) \otimes \big( A \otimes (A \otimes( \cdots (A \otimes A) \cdots ))\big)
  \ar[r]^-{\Phi^{(-n)}} &\left(\cdots \big((\mathrm{end}(V) \otimes A) \otimes A\big) \cdots \right)\otimes A ~,
 }
\end{flalign}
where again the source and target contain $n$ factors of $A$. We shall 
denote the source of this ${}^H\MMM$-isomorphism also
by $\mathrm{end}(V) \otimes A^{\otimes n}$.
\begin{defi}\label{defi:diffop}
Let $A$ be an object in $^H\AAA^{\mathrm{com}}$ and $V$ any object in
${}^{H}_{}{}^{}_{\bol{A}}\MMM_{\bol{A}}^{\mathrm{sym}}$.
The differential operators of order $n\in \bbZ_{\geq0}$ of $V$ is the object $\mathrm{diff}^n(V)$ 
in ${}^H\MMM$ which is defined by the equalizer
\begin{flalign}
\xymatrix{
\mathrm{diff}^n(V) \ar[rr] && \mathrm{end}(V) \ar@<-1ex>[rrrr]_-{0} \ar@<1ex>[rrrr]^-{\zeta ([\, \cdot\, , \, \cdot\, ]^{(n + 1)} \circ \Phi^{(-(n + 1))} )} &&&& \hom(A^{\otimes n}, \mathrm{end}(V))
}
\end{flalign}
in ${}^H\MMM$.
This equalizer can be realized explicitly in terms of the ${}^H\MMM$-subobject
\begin{flalign}\label{eqn:diff}
\mathrm{diff}^n(V) = \mathrm{Ker}\big(\zeta \big([\, \cdot\, , \,\cdot\, ]^{(n + 1)} \circ \Phi^{(-(n + 1))}\big) \big)~
\end{flalign}
of the internal endomorphism object $\mathrm{end}(V)$ in ${}^H\MMM$.
\end{defi}
\begin{rem}\label{rem:diff0}
Comparing Definitions \ref{defi:diffop} and \ref{defi:homAcurrent} we observe that
the order $0$ differential operators $\mathrm{diff}^0(V)$ are 
the internal endomorphisms $\mathrm{end}_A(V)$ in the category ${}^H_{}{}^{}_{A}\MMM^{\mathrm{sym}}_{A}$.
\end{rem}

\begin{lem}\label{lem:diff}
Let $A$ be any object in $^H\AAA^{\mathrm{com}}$ and let $V$
be any object in $ {}^{H}_{}{}^{}_{\bol{A}}\MMM_{\bol{A}}^{\mathrm{sym}} $. 
An $^H\MMM$-subobject $U \subseteq \mathrm{end}(V)$ 
is an $^H\MMM$-subobject of $\mathrm{diff}^n(V)$ if and only if
\begin{flalign}
\big[\big[\cdots \big[ [L ,a_1],a_2 \big], \cdots \big],a_{n + 1}\big] = 0~,
\end{flalign}
for all $L \in U$ and $a_1,a_2, \dots ,a_{n + 1} \in A$.
\end{lem}
\begin{proof}
Denoting by $f := [\, \cdot\, , \,\cdot\, ]^{(n + 1)} \circ \Phi^{(-(n + 1))} : 
\mathrm{end}(V)\otimes A^{\otimes n} \to \mathrm{end}(V)$
and $j : U \to \mathrm{end}(V) $ the inclusion ${}^H\MMM$-morphism,
it follows from Lemma~\ref{lem:evcurry} (ii) that 
$\zeta(f) \circ j =0$ if and only if $f\circ (j\otimes \id)=0$. 
The latter condition is equivalent to $[\, \cdot\, , \,\cdot\, ]^{(n + 1)}\circ 
\big((\cdots ((j \otimes \id)\otimes \id) \cdots ) \otimes \id \big) \circ  \Phi^{(-(n + 1))}= 0$,
and the assertion now follows because $\Phi^{(-(n + 1))}$ is an isomorphism.
\end{proof}

There is an ${}^H\MMM$-subobject relation
$\mathrm{diff}^n(V) \subseteq \mathrm{diff}^m(V) $ for all $n\leq m$, which
immediately follows from Lemma~\ref{lem:diff} and (\ref{eqn:diff}).
These subobject relations give rise to the sequence of ${}^H\MMM$-monomorphisms
\begin{flalign}\label{eqn:inclusions}
\xymatrix{
\mathrm{diff}^0(V) \ar[r] & \mathrm{diff}^1(V) \ar[r] & \mathrm{diff}^2(V) \ar[r] 
&~\cdots~ \ar[r] & \mathrm{diff}^n(V) \ar[r] &~\cdots \ .
}
\end{flalign}

We shall now show that differential operators can be composed with respect to the internal composition.
\begin{propo}\label{propo:diffnbullet}
The internal composition $\bullet : \mathrm{end}(V)\otimes \mathrm{end}(V)\to\mathrm{end}(V)$ 
restricts to an $^H\MMM$-morphism
\begin{flalign}
\bullet: \mathrm{diff}^n(V) \otimes \mathrm{diff}^m(V) \longrightarrow \mathrm{diff}^{n+m}(V)~,
\end{flalign}
for all $n,m \in \bbZ_{\geq0}$.
\end{propo}
\begin{proof}
Restricting $\bullet : \mathrm{end}(V)\otimes \mathrm{end}(V)\to\mathrm{end}(V)$ 
to the corresponding ${}^H\MMM$-subobjects of differential operators yields
an $^H\MMM$-morphism $\bullet: \mathrm{diff}^n(V) \otimes \mathrm{diff}^m(V) \to \mathrm{end}(V)$
and we have to prove that its image lies in $\mathrm{diff}^{n+m}(V)$. As the image of this $^H\MMM$-morphism
is an ${}^H\MMM$-subobject of $\mathrm{end}(V)$, by Lemma~\ref{lem:diff} it is enough to show that
\begin{flalign}
\big[\big[\cdots \big[ [L\bullet L^\prime ,a_1],a_2 \big], \cdots \big],a_{n + m+1}\big] = 0~,
\end{flalign}
for all $L\in \mathrm{diff}^n(V)$, $L^\prime\in \mathrm{diff}^m(V)$ and
$a_1,a_2, \dots ,a_{n + m+1} \in A$. This equality follows by iteratively using the
derivation property of the internal commutator, 
cf.\ item (iii) of Proposition~\ref{prop:bracket}, and applying Lemma~\ref{lem:diff} to $L$ and $L^\prime$.
\end{proof}

Forming the colimit in ${}^H\MMM$ of the diagram given in (\ref{eqn:inclusions}) 
we can define the object  $\mathrm{diff}(V)$ of differential operators on $V$.
This colimit can be represented explicitly as the union of differential operators of all orders $n \in\bbZ_{\geq0}$, i.e.\
\begin{flalign}
\mathrm{diff}(V) = \bigcup_{n \in\bbZ_{\geq0}} \, \mathrm{diff}^n(V) \subseteq \mathrm{end}(V) ~.
\end{flalign}
\begin{cor}\label{cor:diffalg}
The differential operators $\mathrm{diff}(V)$ is an ${}^H\AAA$-subobject of the algebra 
of internal endomorphisms $\mathrm{end}(V)$ (cf.\ Example \ref{defi:internalendalgebra}).
\end{cor}
\begin{proof}
By Proposition~\ref{propo:diffnbullet} the internal composition closes on $\mathrm{diff}(V)$, 
i.e.\ there is an ${}^H\MMM$-morphism
\begin{flalign}
\bullet: \mathrm{diff}(V) \otimes \mathrm{diff}(V) \longrightarrow \mathrm{diff}(V)~.
\end{flalign}
The unit $\eta : I\to \mathrm{end}(V)$ has its image in the degree
$0$ differential operators
because of the calculation
\begin{flalign}
[c\,1,a] = c\,1\bullet \widehat{l}(a) - c\, \widehat{l}\big(R^{(2)}\ra a\big) \bullet \big(R^{(1)}\ra 1\big)
= c\, \widehat{l}(a) - c\, \widehat{l}\big(R^{(2)}\, \epsilon(R^{(1)})\ra a\big) =0
\end{flalign}
and Lemma~\ref{lem:diff}; here we used the normalization $(\epsilon\otimes \id)(R)=1$ of the $R$-matrix.
\end{proof}
\begin{rem}\label{rem:diff}
Combining Lemmas~\ref{lem:der} and~\ref{lem:diff} we see that for any object $A$ in ${}^H\AAA^{\mathrm{com}}$,
$ \mathrm{der}(A) \subseteq \mathrm{diff}^1(A) $ is an ${}^H\MMM$-subobject,
i.e.\ the derivations of $A$ are differential operators of order $1$.
\end{rem}

With the techniques developed above we can now introduce
the notion of a differential calculus in ${}^H\MMM$.
In the following we shall denote by $I[1]$ the object in ${}^H\MMM$
which is obtained by shifting the unit object $I = (k,\ra)$ in $\bbZ$-degree by $1$:
$I[1]_1 = k$ and $I[1]_n =0$, for all $n\neq 1$.
\begin{defi}\label{defi:diffcalc}
Let $H$ be a quasitriangular quasi-Hopf algebra.
A differential calculus $(A,\dd)$ in ${}^H\MMM$ is an object $A$ in ${}^H\AAA^{\mathrm{com}}$
together with an ${}^H\MMM$-morphism $\dd : I[1]\to  \mathrm{der}(A)$ 
which is nilpotent in the sense that the composition of ${}^H\MMM$-morphisms
\begin{flalign}
\xymatrix{
I[1]\otimes I[1] \ar[r]^-{\dd\otimes \dd}& \mathrm{der}(A)\otimes \mathrm{der}(A)
\ar[r]^-{}& \mathrm{diff}(A)\otimes \mathrm{diff}(A)\ar[r]^-{\bullet} & \mathrm{diff}(A)
}
\end{flalign}
is $0$; here the second arrow is defined using Remark \ref{rem:diff}.
\end{defi}
\begin{rem}
Given a differential calculus $(A,\dd)$ in ${}^H\MMM$ there is a distinguished
$H$-invariant derivation of $\bbZ$-degree $1$, which is given by
$\dd(1)\in \mathrm{der}(A)$ and is called the differential.
\end{rem}
\begin{ex}\label{ex:classicaldiffcalc}
Building upon Example \ref{ex:classicalGMan}, examples
of differential calculi are provided by the exterior algebras of
differential forms $\Omega^\sharp(M)$ on $G$-manifolds $M$, equipped with the de~Rham differential,
and cochain twist quantizations thereof. See Proposition~\ref{propo:diffcalctwisting}
for details on the twist deformation quantization of differential calculi.
\end{ex}

%%%%%%%%%%%%%%%%%%%%%%%%%%%%%%%%%%%%%%%%%%%%%%%%%%%%%%%
%%%%%%%%%%%%%%%%%%%%%%%%%%%%%%%%%%%%%%%%%%%%%%%%%%%%%%%

\subsection{Cochain twisting of differential operators and calculi}
The cochain twist deformation quantization functor preserves 
differential operators and differential calculi. 
\begin{propo}\label{propo:diffcochain}
Let $A$ be any object in  ${}^H\AAA^{\mathrm{com}}$,
$V$ any object in $ {}^H_{}{}^{}_{A}\MMM^{\mathrm{sym}}_{A} $ and $F$ any
cochain twisting element based on $H$. Then the coherence map 
$\gamma: \mathrm{end}_F(\FF(V))\to \FF(\mathrm{end}(V))$ restricts to an ${}^{H_F}_{}\MMM$-isomorphism
\begin{flalign}
\gamma: \mathrm{diff}^n_F(\FF(V)) \longrightarrow \FF(\mathrm{diff}^n(V)) ~,
\end{flalign}
for all $n\in\bbZ_{\geq0}$.
\end{propo}
\begin{proof}
The proof is analogous to the proof of Proposition~\ref{propo:cochainder}.
\end{proof}
\begin{propo}\label{propo:diffcalctwisting}
Let $\big(A, \dd : I[1]\to \mathrm{der}(A)\big)$ be any differential 
calculus in $^H\MMM$ and let $F$ be any cochain twisting element based on $H$. 
Then $\FF(A)$ together with the ${}^{H_F}\MMM$-morphism
\begin{flalign}\label{eqn:twisteddifferential}
\dd_F := \gamma^{-1}\circ \FF(\dd) \circ \psi : I_F[1]\longrightarrow \mathrm{der}_F(\FF(A))
\end{flalign}
is a differential calculus in $^{H_F}\MMM$, where $\psi$ is the coherence morphism in (\ref{eqn:coherencemapstensor2}).
\end{propo}
\begin{proof}
By Proposition~\ref{propo:cochainder}, the target 
of $\dd_F$ is as claimed in (\ref{eqn:twisteddifferential}). Moreover, $\dd_F$
is nilpotent (in $\mathrm{diff}_F(\FF(A))$) because of the short calculation
\begin{flalign}
\nn \dd_F(c)\bullet_F\dd_F(c^\prime\, ) &= \gamma^{-1}\left(\big(F^{(-1)}\ra \gamma(\dd_F(c))\big)\bullet 
\big(F^{(-2)}\ra \gamma(\dd_F(c^\prime\, ))\big)\right) \\[4pt]
\nn &= \gamma^{-1}\left(\dd\big(F^{(-1)}\ra c\big)\bullet 
\dd \big(F^{(-2)}\ra c^\prime\, \big)\right)\\[4pt]
&=\gamma^{-1}\big(\dd(c)\bullet \dd(c^\prime\, ) \big)=0~,
\end{flalign}
for all $c,c^\prime\in I[1]$.
In the first equality we have used~\cite[Proposition~2.16]{Barnes:2014}, in the second equality
the definition of $\dd_F$ and the $H$-equivariance of $\dd$, in the third
equality the normalization of the cochain twist, and in the last equality the nilpotency of $\dd$.
\end{proof}

%%%%%%%%%%%%%%%%%%%%%%%%%%%%%%%%%%%%%%%%%%%%%%%%%%%%%%%
%%%%%%%%%%%%%%%%%%%%%%%%%%%%%%%%%%%%%%%%%%%%%%%%%%%%%%%

\section{\label{sec:connections}Connections}
For a given differential calculus $(A, \dd)$ in $^H\MMM$, we shall
develop the notion of connections on objects in $ {}^H_{}{}^{}_{A}\MMM^{\mathrm{sym}}_{A} $
by again using universal constructions in the category ${}^H\MMM$. We will show that
connections of objects $V,W$ in $ {}^H_{}{}^{}_{A}\MMM^{\mathrm{sym}}_{A} $ can be
canonically lifted to connections on the tensor product object $V\otimes_A W$ and
on the internal hom-object $\hom_A(V,W)$.
Throughout this section $H$ is an arbitrary quasitriangular quasi-Hopf algebra.

%%%%%%%%%%%%%%%%%%%%%%%%%%%%%%%%%%%%%%%%%%%%%%%%%%%%%%%
%%%%%%%%%%%%%%%%%%%%%%%%%%%%%%%%%%%%%%%%%%%%%%%%%%%%%%%

\subsection{Connections on symmetric bimodules}
Let $\big(A, \dd : I[1]\to \mathrm{der}(A)\big)$ be a differential calculus 
in $^H\MMM$. Connections on an object $V$ in ${}^H_{}{}^{}_{A}\MMM^{\mathrm{sym}}_{A}$
are distinguished differential operators of order $1$ which satisfy a Leibniz rule 
with respect to the $^H\MMM$-morphism $\dd$. We shall again formalize this algebraic property in terms
of an equalizer in ${}^H\MMM$.
Denoting by $\times$ the categorical product in $^H\MMM$ and recalling
that $I[1]$ denotes the shifted unit object in ${}^H\MMM$, there are
two parallel ${}^H\MMM$-morphisms
\begin{flalign}\label{eqn:conarrows}
\xymatrix{
\big(\mathrm{end}(V)\times I[1]\big)\otimes A 
\ar@<-1ex>[rrrr]_-{\widehat{l}\circ \ev \circ (\dd\otimes\id)\circ (\mathrm{pr}_2\otimes \id)} 
\ar@<1ex>[rrrr]^-{[\,\cdot\,,\,\cdot\,]\circ (\mathrm{pr}_1\otimes \id)} 
 &&&& \mathrm{end}(V)}~,
\end{flalign}
where $\mathrm{pr}_{1} : \mathrm{end}(V)\times I[1]\to \mathrm{end}(V)$
and $\mathrm{pr}_2: \mathrm{end}(V)\times I[1] \to I[1]$ are the projection ${}^H\MMM$-morphisms.
The upper arrow in (\ref{eqn:conarrows}) is the mapping
\begin{subequations}
\begin{flalign}
(L,c)\otimes a \longmapsto  [L,a]
\end{flalign}
and the lower arrow is the mapping
\begin{flalign}
(L,c)\otimes a \longmapsto \widehat{l}\big(\ev(\dd(c)\otimes a)\big)~.
\end{flalign}
\end{subequations}
\begin{defi}
Let $(A,\dd)$ be a differential calculus in $^H\MMM$ and $V$ any object in
${}^{H}_{}{}^{}_{\bol{A}}\MMM_{\bol{A}}^{\mathrm{sym}}$.
The connections of $V$ is the object $\mathrm{con}(V)$ 
in ${}^H\MMM$ which is defined by the equalizer
\begin{flalign}\label{eqn:k1connections}
\xymatrix{
\mathrm{con}(V)\ar[rr]&& \mathrm{end}(V) \times I[1] \ar@<-1ex>[rrrr]_-{\zeta(\widehat{l}\circ \ev \circ (\dd\otimes\id)\circ (\mathrm{pr}_2\otimes \id))} \ar@<1ex>[rrrr]^-{\zeta([\,\cdot\,,\,\cdot\,]\circ (\mathrm{pr}_1\otimes \id))}  &&&& \hom(A, \mathrm{end}(V))
}
\end{flalign}
in ${}^H\MMM$. This equalizer can be realized explicitly in terms of the ${}^H\MMM$-subobject
\begin{flalign}
\mathrm{con}(V) = \mathrm{Ker}\Big(\zeta \big([\,\cdot\,,\,\cdot\,]\circ (\mathrm{pr}_1\otimes \id) -
\widehat{l}\circ \ev \circ (\dd\otimes\id)\circ (\mathrm{pr}_2\otimes \id) \big) \Big)~
\end{flalign}
of the object $\mathrm{end}(V)\times I[1]$ in ${}^H\MMM$.
\end{defi}

\begin{lem}\label{lem:con}
Let $(A,\dd)$ be any differential calculus in ${}^H\MMM$ 
and let $V$ be any object in $ {}^{H}_{}{}^{}_{\bol{A}}\MMM_{\bol{A}}^{\mathrm{sym}} $. 
An ${}^H\MMM$-subobject $U \subseteq \mathrm{end}(V)\times I[1]$ 
is an ${}^H\MMM$-subobject of $\mathrm{con}(V)$ if and only if
\begin{flalign}\label{eqn:concondition}
[L, a] = \widehat{l}\big(\ev(\dd(c) \otimes a) \big)~,
\end{flalign}
for all $(L,c) \in U$ and $a\in A$.
\end{lem}
\begin{proof}
Denoting by $f := [\,\cdot\,,\,\cdot\,]\circ (\mathrm{pr}_1\otimes \id) -
\widehat{l}\circ \ev \circ (\dd\otimes\id)\circ (\mathrm{pr}_2\otimes \id) : 
(\mathrm{end}(V)\times I[1])\otimes A \to \mathrm{end}(V)$
and $j : U \to \mathrm{end}(V)\times I[1] $ the inclusion ${}^H\MMM$-morphism,
we have to show that $\zeta(f) \circ j =0$ if and only if $f\circ (j\otimes \id)=0$.
This is a consequence of item (ii) of Lemma~\ref{lem:evcurry}.
\end{proof}
\begin{rem}
By Lemma~\ref{lem:con}, any element $(L,c)\in \mathrm{con}(V)$ satisfies
the condition (\ref{eqn:concondition}) for all $a\in A$. In particular,
the $\bbZ$-degree $1$ elements $\nabla = (L,1) \in \mathrm{con}(V)$ satisfy the 
Leibniz rule with respect to the differential $\dd(1)$. Hence, our notion of connections
contains the standard notion of connections as distinguished points.
It is important to notice that our definition has the advantage
that $\mathrm{con}(V)$ is by construction an object in ${}^H\MMM$ while
the subset of all ordinary connections $\nabla = (L,1)\in\mathrm{con}(V)$ 
is just an affine space over the $k$-module of all $\bbZ$-degree $1$ elements $(L,0)\in\mathrm{con}(V) $,
hence it is not an object in ${}^H\MMM$. 
\end{rem}

Finally, we prove an important structural result for connections.
\begin{propo}
Let $(A, \dd)$ be any differential calculus in ${}^H\MMM$ and let $V$ be any object in
${}^H_{}{}^{}_{A}\MMM^{\mathrm{sym}}_{A}$.  Then $\mathrm{con}(V)$
 is an ${}^H\MMM$-subobject of $\mathrm{diff}^1(V) \times I[1]$.
\end{propo}
\begin{proof}
The object $\mathrm{con}(V)$ is by construction an ${}^H\MMM$-subobject of $\mathrm{end}(V)\times I[1]$
and hence the image of $\mathrm{pr}_1 : \mathrm{con}(V) \to \mathrm{end}(V)$ is an ${}^H\MMM$-subobject 
of $\mathrm{end}(V)$. We have
\begin{flalign}
[[L,a],a^\prime\, ] = \big[\, \widehat{l} \big(\ev(\dd(c) \otimes a) \big), a^\prime\, \big] = 0~,
\end{flalign}
for all $(L,c)\in \mathrm{con}(V)$, which by using Lemma~\ref{lem:diff} shows that the image of
$\mathrm{pr}_1 : \mathrm{con}(V) \to \mathrm{end}(V)$ is an ${}^H\MMM$-subobject 
of $\mathrm{diff}^{1}(V)$ and hence that $\mathrm{con}(V)$ is an ${}^H\MMM$-subobject 
of $\mathrm{diff}^{1}(V)\times I[1]$.
\end{proof}

%%%%%%%%%%%%%%%%%%%%%%%%%%%%%%%%%%%%%%%%%%%%%%%%%%%%%%%
%%%%%%%%%%%%%%%%%%%%%%%%%%%%%%%%%%%%%%%%%%%%%%%%%%%%%%%

\subsection{Connections on tensor products}
We shall now develop a lifting prescription for connections to tensor products of objects
in ${}^H_{}{}^{}_{A}\MMM^{\mathrm{sym}}_{A}$.
Let us first notice that for any two objects $V,W$ in ${}^H_{}{}^{}_{A}\MMM^{\mathrm{sym}}_{A}$ 
there are two ${}^H\MMM$-morphisms given by the compositions
\begin{subequations}
\begin{flalign}
\xymatrix{
\mathrm{end}(V)\ar[r]^-{\rho^{-1}} & \mathrm{end}(V)\otimes I\ar[r]^-{\id\otimes\eta} & 
\mathrm{end}(V)\otimes \mathrm{end}(W)\ar[r]^-{\obultimes} & \mathrm{end}(V\otimes W)
}
\end{flalign}
and
\begin{flalign}
\xymatrix{
\mathrm{end}(W)\ar[r]^-{\lambda^{-1}} & I \otimes \mathrm{end}(W)
\ar[r]^-{\eta\otimes \id} & \mathrm{end}(V)\otimes \mathrm{end}(W)\ar[r]^-{\obultimes} & \mathrm{end}(V\otimes W) \ .
}
\end{flalign}
\end{subequations}
These ${}^H\MMM$-morphisms are given explicitly by the mappings
\begin{flalign}
L\longmapsto L\obultimes 1 \qquad\text{and}\qquad L^\prime \longmapsto 1\obultimes L^\prime~,
\end{flalign}
respectively.
\begin{defi}\label{defi:obulplus}
For any two objects $V,W$ in ${}^H_{}{}^{}_{A}\MMM^{\mathrm{sym}}_{A}$
we define the $^H\MMM$-morphism
\begin{flalign}
\nn \obulplus: \big(\mathrm{end}(V) \times I[1]\big) \times \big(\mathrm{end}(W) \times I[1]\big) &\longrightarrow \mathrm{end}(V \otimes W) \times I[1]~,\\[4pt]
\big((L, c),(L^\prime, c^\prime\, )\big) &\longmapsto \big(L\obultimes 1 +1\obultimes L^\prime, c\big) ~.
\end{flalign}
\end{defi}

In order to prove that $\obulplus$ restricts
to connections, i.e.\ to an ${}^H\MMM$-morphism $\obulplus : \mathrm{con}(V)\times \mathrm{con}(W) \to \mathrm{con}(V\otimes W)$,
we require the following technical lemma.
\begin{lem}\label{lem:fortensorcon}
Let $A$ be an object in $^H\AAA^{\mathrm{com}}$ 
and let $ V,W$ be any two objects in ${}^H_{}{}^{}_{A}\MMM^{\mathrm{sym}}_{A}$.
\begin{itemize}
\item[(i)] 
Recalling the ${}^H\MMM$-morphisms $\widehat{l}_V : A \to \mathrm{end}(V)$
and $\widehat{l}_{V\otimes W} : A\to\mathrm{end}(V\otimes W)$
given in (\ref{eqn:leftAactionvsrepresentation1}), one has
\begin{flalign}
\widehat{l}_{V \otimes W} (a) = \widehat{l}_{V} (a) \obultimes 1~,
\end{flalign}
for all $a \in A$. 
\item[(ii)] For any $L,K \in \mathrm{end}(V)$ and $L^\prime \in \mathrm{end}(W)$,
one has
%\begin{subequations}
\begin{flalign}
[K \obultimes 1, L \obultimes 1] = [K ,L ] \obultimes 1 \qquad \mbox{and} \qquad 
[1\obultimes L^\prime, L \obultimes 1] = 0~.
\end{flalign}
%\end{subequations}
\item[(iii)] For any $L \in \mathrm{end}(V)$,  $L^\prime \in \mathrm{end}(W) $ and
$a \in A$, one has
\begin{flalign}
\big[L\obultimes 1 +  1\obultimes L^\prime , a\big] =  [L, a ] \obultimes 1~.
\end{flalign}
\end{itemize}
\end{lem}
\begin{proof}
Let us first prove item (i). By definition of $\widehat{l}_{V \otimes W}$ we have
\begin{flalign}
\ev\big(\, \widehat{l}_{V \otimes W} (a)\otimes (v \otimes w)\big) =
a\,(v\otimes w)~,
\end{flalign}
for all $v \in V, w \in W$ and $ a \in A$. On the other hand, using~\cite[Equation (5.7a)]{Barnes:2014}
we have
\begin{flalign}
\ev\big((\, \widehat{l}_V(a)\obultimes 1) \otimes (v\otimes w)\big)=
\big((\phi^{(-1)}\ra a)\,(\phi^{(-2)}\ra v)\big)\otimes (\phi^{(-3)}\ra w) = a\,(v\otimes w)~,
\end{flalign}
for all $v \in V, w \in W$ and $ a \in A$. The assertion then follows by using Lemma~\ref{lem:evcurry} 
(i) and invertibility of $\zeta$.
Item (ii) follows immediately from Lemma~\ref{lem:tensorproductproperties}.
Item (iii) is a consequence of item (i) and (ii), as
\begin{flalign}
\nn \big[L\obultimes 1 +  1\obultimes L^\prime , a\big] &= 
\big[L\obultimes 1 +  1\obultimes L^\prime , \widehat{l}_{V\otimes W}(a)\big]\\[4pt]
&=\big[L\obultimes 1 +  1\obultimes L^\prime , \widehat{l}_V(a)\obultimes 1\big]= [L, a ] \obultimes 1~,
\end{flalign}
and the assertion follows.
\end{proof}

\begin{propo}
Let $(A, \dd)$ be a differential calculus in $^H\MMM$ and let $ V,W$ be two 
objects in ${}^H_{}{}^{}_{A}\MMM^{\mathrm{sym}}_{A}$. Then $\obulplus$ restricts to an $^H\MMM$-morphism
\begin{flalign}\label{eqn:obulpluscon}
\obulplus: \mathrm{con}(V) \times  \mathrm{con}(W) &\longrightarrow \mathrm{con}(V \otimes W) ~.
\end{flalign}
\end{propo}
\begin{proof}
We have to show that the image of $\obulplus: \mathrm{con}(V) \times  \mathrm{con}(W) 
\rightarrow \mathrm{end}(V \otimes W) \times I[1] $ is an ${}^H\MMM$-subobject of
$\mathrm{con}(V\otimes W)$.  Using Lemma~\ref{lem:con} this can be shown by the computation
\begin{flalign}
\big[L\obultimes 1 + 1\obultimes L^\prime, a\big]=[L,a]\obultimes 1 = \widehat{l}_V\big(\ev(\dd(c)\otimes a)\big)\obultimes 1
= \widehat{l}_{V\otimes W}\big(\ev(\dd(c)\otimes a)\big)~,
\end{flalign}
for all $(L,c)\in \mathrm{con}(V)$, $(L^\prime,c^\prime\, )\in \mathrm{con}(W)$ and $a\in A$.
In the first equality we used item~(iii) and in the last equality item~(i) of Lemma~\ref{lem:fortensorcon}.
\end{proof}

The ${}^H\MMM$-morphism (\ref{eqn:obulpluscon}) describes the construction of
connections on the object $V\otimes W$ but not on the object $V\otimes_A W$, 
which is obtained by using the correct monoidal functor 
$\otimes_A$ in ${}^H_{}{}^{}_{A}\MMM^{\mathrm{sym}}_{A}$. As $V\otimes_A W$ can be
obtained by taking a quotient of $V\otimes W$ (cf.~(\ref{eqn:otimesA})), we may ask 
if (\ref{eqn:obulpluscon}) induces an ${}^H\MMM$-morphism with target given by
$\mathrm{con}(V\otimes_A W)$. For this to hold true, we have to restrict
the source of (\ref{eqn:obulpluscon}) to the fibred product $\mathrm{con}(V) \times_{I[1]} \mathrm{con}(W) $
 given by the pullback
\begin{flalign}\label{eqn:fibered}
\xymatrix{
\ar[d] \mathrm{con}(V) \times_{I[1]} \mathrm{con}(W) \ar[r] & \mathrm{con}(W) \ar[d]^-{\mathrm{pr}_2}\\
\mathrm{con}(V) \ar[r]_-{\mathrm{pr}_2} & I[1]
}~
\end{flalign}
 in the category $^H\MMM$. Then $\mathrm{con}(V) \times_{I[1]} \mathrm{con}(W)$ is the ${}^H\MMM$-subobject
of $\mathrm{con}(V) \times \mathrm{con}(W)$ with elements given by pairs
$((L,c),(L^\prime ,c^\prime\, ))$ such that $c=c^\prime$.
We can now state one of the main results of this section.
\begin{theo}\label{theo:sumcon}
Let $(A, \dd)$ be a differential calculus in $^H\MMM$ and let $ V,W $ be two
objects in ${}^H_{}{}^{}_{A}\MMM^{\mathrm{sym}}_{A}$. Then
$\obulplus$ induces an ${}^H\MMM$-morphism
\begin{flalign}
\obulplus: \mathrm{con}(V) \times_{I[1]} \mathrm{con}(W) &\longrightarrow \mathrm{con}(V \otimes_A W) ~.
\end{flalign}
\end{theo}
\begin{proof}
Let $((L,c),(L^\prime,c))\in \mathrm{con}(V) \times_{I[1]} \mathrm{con}(W)$ be an arbitrary element.
Applying $\obulplus$ gives the element 
\begin{flalign}
\big(L\obultimes 1 + 1\obultimes L^\prime , c \big) \in 
\mathrm{con}(V\otimes W)\subseteq \mathrm{end}(V\otimes W)\times I[1]~,
\end{flalign}
where we regard $K := L\obultimes 1 + 1\obultimes L^\prime$ simply as a $k$-linear map
$K : V\otimes W \to V\otimes W$. We have to prove that
$K$ descends to a well-defined $k$-linear map $K : V\otimes_A W \to V\otimes_A W$ on the quotient (\ref{eqn:otimesA}). Denoting by $\pi : V\otimes W\to V\otimes_A W$ the quotient map,
this amounts to showing that
\begin{flalign}\label{eqn:tmpcheckprodcon}
\pi\circ K\Big(\big(v\,a\big)\otimes w - (\phi^{(1)}\ra v)\otimes\big((\phi^{(2)}\ra a)\,(\phi^{(3)}\ra w)\big)\Big) =0~,
\end{flalign}
for all $v\in V$, $w\in W$ and $a\in A$. Let us for the moment consider the case of trivial associator
$\phi = 1\otimes 1\otimes 1$. Then the equality (\ref{eqn:tmpcheckprodcon}) can be easily verified on homogeneous elements
by using
\begin{subequations}
\begin{flalign}
\nn \pi\circ (L\obultimes 1)\big((v\,a) \otimes w\big) &= L(v\,a) \otimes_A w \\[4pt]
\nn &= (-1)^{\vert v\vert\,\vert a\vert}~L\big((R^{(2)}\ra a)\,(R^{(1)}\ra v)\big) \otimes_A w\\[4pt]
\nn &=(-1)^{(\vert v\vert+ \vert L\vert)\,\vert a\vert}~\big(\widetilde{R}{}^{(2)}\,R^{(2)}\ra a\big)\, \big(\widetilde{R}{}^{(1)}\ra L\big)\big(R^{(1)}\ra v\big) \otimes_A w\\
\nn &~~~~~+ (-1)^{\vert v\vert\,\vert a\vert}~\big(\dd(c)\big)\big(R^{(2)}\ra a\big)\, \big(R^{(1)}\ra v\big) \otimes_A w\\[4pt]
\nn &= L(v)\otimes_A \big(a\,w \big) + (-1)^{\vert v\vert\,\vert L\vert} ~v\otimes_A \big( \big(\dd(c)\big)(a)\,w\big)\\[4pt]
&= \pi\circ (L\obultimes 1)\big(v\otimes (a\,w) \big) + (-1)^{\vert v\vert\,\vert L\vert} ~v\otimes_A \big(\big(\dd(c)\big)(a)\,w \big)
\end{flalign}
and
\begin{flalign}
\nn \pi\circ (1\obultimes L^\prime\, )\big((v\,a) \otimes w\big)&= 
(-1)^{(\vert v\vert +\vert a\vert)\,\vert L^\prime\vert}~\big( (R^{(2)}\ra v)\,(\widetilde{R}{}^{(2)}\ra a) \big)\otimes_A
\big(\widetilde{R}{}^{(1)}\, R^{(1)} \ra L^\prime\, \big)\big(w\big)\\[4pt]
\nn &=(-1)^{\vert v\vert\,\vert L^\prime\vert}~\left(\big(R^{(2)}\ra v \big)\otimes_A \big(R^{(1)}\ra L^\prime\, \big)\big(a\,w\big) - v\otimes_A \big( \big(\dd(c)\big)(a)\,w \big) \right)\\[4pt]
&= \pi\circ (1\obultimes L^\prime\, )\big(v\otimes (a\,w) \big)  - (-1)^{\vert v\vert\,\vert L\vert} ~v\otimes_A \big(\big(\dd(c)\big)(a)\,w \big)~,
\end{flalign}
\end{subequations}
where in the last equality we used $\vert L\vert = \vert L^\prime\,\vert$.
The equality (\ref{eqn:tmpcheckprodcon}) also holds for the case of nontrivial associators, however
the corresponding calculation is much more lengthy and involved, and hence we will not write it out in detail.
\end{proof}

The following result allows us to consistently lift connections to tensor 
products of an arbitrary (finite) number of objects in ${}^H_{}{}^{}_{A}\MMM^{\mathrm{sym}}_{A}$.
\begin{theo}
Let $(A, \dd)$ be a differential calculus in $^H\MMM$ and let $ V,W,X $ be three
objects in ${}^H_{}{}^{}_{A}\MMM^{\mathrm{sym}}_{A}$.  Then the ${}^H\MMM$-diagram
\begin{flalign}
\xymatrix{
\ar[d]_-{\obulplus \circ \big(\id\times \obulplus\big)}\mathrm{con}(V)\times_{I[1]} \mathrm{con}(W)\times_{I[1]}\mathrm{con}(X)\ar[rr]^-{\obulplus \circ \big(\obulplus\times \id\big)} && 
\mathrm{con}\big((V\otimes_A W)\otimes_A X\big)\ar[dll]^-{~~~~\Phi \circ (\,\cdot\,)\circ \Phi^{-1}}\\
\mathrm{con}\big(V\otimes_A (W\otimes_A X)\big)&&
}
\end{flalign}
commutes.
\end{theo}
\begin{proof}
Let $\big((L,c),(L^\prime,c),(L^{\prime\prime},c)\big)\in 
\mathrm{con}(V)\times_{I[1]} \mathrm{con}(W)\times_{I[1]}\mathrm{con}(X)$ be an arbitrary element.
Applying $\obulplus \circ (\obulplus\times \id)$ yields
\begin{subequations}
\begin{flalign}
\obulplus \circ (\obulplus\times \id) \left(\big((L,c),(L^\prime,c),(L^{\prime\prime},c)\big)\right)
=\big( (L\obultimes 1)\obultimes 1 + (1\obultimes L^\prime\, )\obultimes 1 + (1\obultimes 1)\obultimes L^{\prime\prime},c\big)
\end{flalign}
while applying $\obulplus \circ (\id \times \obulplus)$  yields
\begin{flalign}
\obulplus \circ (\id \times \obulplus) \left(\big((L,c),(L^\prime,c),(L^{\prime\prime},c)\big)\right)
=\big( L\obultimes (1 \obultimes 1) + 1\obultimes ( L^\prime\obultimes 1) + 1\obultimes (1 \obultimes L^{\prime\prime} \, ),c\big)~.
\end{flalign}
\end{subequations}
The assertion then follows by using Lemma~\ref{lem:obultimesphi}.
\end{proof}

%%%%%%%%%%%%%%%%%%%%%%%%%%%%%%%%%%%%%%%%%%%%%%%%%%%%%%%
%%%%%%%%%%%%%%%%%%%%%%%%%%%%%%%%%%%%%%%%%%%%%%%%%%%%%%%

\subsection{Connections on internal homomorphisms}
We shall now develop a lifting prescription for connections to the internal hom-objects
in ${}^H_{}{}^{}_{A}\MMM^{\mathrm{sym}}_{A}$. Let $(A,\dd)$ be a differential calculus
in ${}^H\MMM$ and $V,W$ two objects in ${}^H_{}{}^{}_{A}\MMM^{\mathrm{sym}}_{A}$.
Then there are two $^H\MMM$-morphisms
\begin{subequations}
\begin{flalign}
\LLL := \zeta(\bullet): \mathrm{end}(W) &\longrightarrow  \mathrm{end}(\hom(V,W))~, \\[4pt]
\RRR := \zeta(\bullet \circ \tau): \mathrm{end}(V)  &\longrightarrow \mathrm{end}(\hom(V,W))~.
\end{flalign}
\end{subequations}
\begin{defi}\label{defi:coninthom}
For any two objects $V,W$ in ${}^H_{}{}^{}_{A}\MMM^{\mathrm{sym}}_{A}$
we define the $^H\MMM$-morphism
\begin{flalign}
 \nn \DDD : \big( \mathrm{end}(W) \times I[1]\big) \times \big( \mathrm{end}(V) \times I[1]\big) &\longrightarrow \mathrm{end}(\hom(V,W)) \times I[1]~,\\[4pt]
\big((L^\prime,c^\prime\, ), (L, c)\big) &\longmapsto \big(\LLL(L^\prime\, ) - \RRR(L), c^\prime\, \big)~.
\end{flalign}
\end{defi}

We shall require the following two technical lemmas. 
\begin{lem}\label{lem:homa}
Let $A$ be an object in $^H\AAA^{\mathrm{com}}$ and let $V, W$ be any two objects in 
${}^H_{}{}^{}_{A}\MMM^{\mathrm{sym}}_{A}$. 
Recalling the ${}^H\MMM$-morphisms $\widehat{l}_W : A \to \mathrm{end}(W)$
and $\widehat{l}_{\hom(V,W)} : A\to\mathrm{end}(\hom(V,W))$
given in (\ref{eqn:leftAactionvsrepresentation1}), one has
\begin{flalign}
\widehat{l}_{\hom(V,W)}(a) = \LLL\big(\, \widehat{l}_W(a) \big)~,
\end{flalign}
for all $a \in A$.
\end{lem}
\begin{proof}
Recalling (\ref{eqn:lrhomA}) and using naturality of the currying bijection yields
\begin{flalign}
\nn \widehat{l}_{\hom(V,W)}(a) &= \zeta\big(\bullet \circ(\, \widehat{l}_W \otimes \id) \big)(a) \\[4pt] \nn &= \zeta^{}\big(\Hom^{}_{{}^H\MMM}(\, {\widehat{l}_W}^{ \ \mathrm{op}}\otimes \id^{\mathrm{op}}, \id)(\bullet)\big)(a)\\[4pt] \nn
&= \Hom^{}_{{}^H\MMM}\big(\, {\widehat{l}_W}^{ \ \mathrm{op}},\hom(\id^{\mathrm{op}},\id) \big)\big(\zeta^{}(\bullet)\big)(a) \\[4pt] &= \zeta(\bullet)\big( \, \widehat{l}_W(a) \big)= \LLL\big(\, \widehat{l}_W(a) \big)~,
\end{flalign}
for all $a\in A$.
\end{proof}
\begin{lem}\label{lem:LLLRRR}
Let $V, W$ be two objects in ${}^H_{}{}^{}_{A}\MMM^{\mathrm{sym}}_{A}$. Then
\begin{subequations}
\begin{flalign}
\label{eqn:1st} \LLL(L) \bullet \LLL(L^\prime\,) &= \LLL(L \bullet L^\prime\, )~,\\[4pt]
\label{eqn:3rd} \RRR(K) \bullet \LLL(L) &= (-1)^{\vert K \vert\,  \vert L \vert}\, \LLL\big(R^{(2)} \ra L\big) \bullet \RRR\big(R^{(1)} \ra K \big)~,
\end{flalign}
\end{subequations}
for all homogeneous $L,L^\prime\in \mathrm{end}(W)$ and $K\in\mathrm{end}(V)$.
\end{lem}
\begin{proof}
First, let us notice that both sides of (\ref{eqn:1st}) can be regarded as 
${}^H\MMM$-morphisms $\mathrm{end}(W)\otimes \mathrm{end}(W)\to \mathrm{end}(\hom(V,W))$: The morphism on the left-hand side is given by $ \bullet\circ (\LLL\otimes \LLL)$
and on the right-hand side by $\LLL\circ \bullet$. By invertibility of
the natural currying bijections, these two morphisms agree if and only $\zeta^{-1}(\bullet\circ (\LLL\otimes \LLL))
=\zeta^{-1}(\LLL\circ \bullet)$ as morphisms $(\mathrm{end}(W)\otimes \mathrm{end}(W))\otimes \hom(V,W)\to\hom(V,W)$.
This can be shown by using Lemma~\ref{lem:evcurry}~(i) and the calculation
\begin{flalign}
\nn \zeta^{-1}(\bullet\circ (\LLL\otimes \LLL))\big((L\otimes L^\prime\, )\otimes M\big)
&=\ev\big((\LLL(L)\bullet \LLL(L^\prime\, )) \otimes M\big)\\[4pt]
\nn &= \ev\left(\LLL(\phi^{(1)}\ra L)\otimes \ev\big(\LLL(\phi^{(2)}\ra L^\prime\,) \otimes (\phi^{(3)}\ra M)\big)\right)\\[4pt]
\nn &= (\phi^{(1)}\ra L)\bullet\big((\phi^{(2)}\ra L^\prime\, ) \bullet (\phi^{(3)}\ra M)\big) \\[4pt]
\nn &=(L\bullet L^\prime\, )\bullet M \\[4pt] 
\nn &= \ev\big(\LLL(L\bullet L^\prime\, )\otimes M\big)\\[4pt]
& = \zeta^{-1}(\LLL\circ \bullet)\big((L\otimes L^\prime\, )\otimes M\big)~,
\end{flalign}
for all $L,L^\prime\in\mathrm{end}(W)$ and $M\in\hom(V,W)$, where we have also used Lemma~\ref{lem:compproperties}.
The equality (\ref{eqn:3rd}) can be shown similarly.
\end{proof}

\begin{propo}\label{propo:conhom}
Let $(A, \dd)$ be a differential calculus in $^H\MMM$ and let $ V,W$ be two 
objects in ${}^H_{}{}^{}_{A}\MMM^{\mathrm{sym}}_{A}$. Then $\DDD$ restricts to an $^H\MMM$-morphism
\begin{flalign}
\DDD: \mathrm{con}(W) \times \mathrm{con}(V) \longrightarrow \mathrm{con}(\hom(V,W))~.
\end{flalign}
\end{propo}
\begin{proof}
We have to show that the image of $\DDD: \mathrm{con}(W) \times  \mathrm{con}(V) 
\rightarrow \mathrm{end}(\hom(V,W)) \times I[1] $ is an ${}^H\MMM$-subobject of
$\mathrm{con}(\hom(V, W))$.  Using Lemma~\ref{lem:con} this can be shown by the computation
\begin{flalign}
\nn \big[\LLL(L^\prime\, ) - \RRR(L), a\big] &= \big[\LLL(L^\prime\, ) - \RRR(L), \widehat{l}_{\hom(V,W)}(a)\big]\\[4pt] 
\nn &= \big[\LLL(L^\prime\, ) - \RRR(L), \LLL(\, \widehat{l}_W(a))\big] \\[4pt]
\nn &= \LLL\big([L^\prime, a]\big)\\[4pt]
\nn &= \LLL\big(\, \widehat{l}_W(\ev (\dd(c^\prime \, ) \otimes a)) \big)\\[4pt]
&= \widehat{l}_{\hom(V,W)}\big(\ev (\dd(c^\prime \, ) \otimes a) \big)~,
\end{flalign}
for all $(L^\prime,c^\prime\, )\in\mathrm{con}(W)$, $(L,c)\in \mathrm{con}(V)$  and $ a \in A $. In the second and last
equality we have used Lemma~\ref{lem:homa} and in the third equality we have used Lemma~\ref{lem:LLLRRR}.
\end{proof}

Restricting the source of $\DDD$ to the fibred product $\mathrm{con}(W) \times_{I[1]} \mathrm{con}(V)$ we obtain
a lifting prescription of connections to the internal hom-objects $\hom_A(V,W)$
in the category ${}^H_{}{}^{}_{A}\MMM^{\mathrm{sym}}_{A}$.
\begin{theo}\label{theo;adjointcon}
Let $(A, \dd)$ be a differential calculus in $^H\MMM$ and let $ V,W $ be two
objects in ${}^H_{}{}^{}_{A}\MMM^{\mathrm{sym}}_{A}$. Then
$\DDD$ induces an ${}^H\MMM$-morphism
\begin{flalign}
\DDD: \mathrm{con}(W) \times_{I[1]} \mathrm{con}(V) &\longrightarrow \mathrm{con}(\hom_A(V,W))~.
\end{flalign}
\end{theo}
\begin{proof}
Let $((L^\prime,c),(L,c))\in \mathrm{con}(W)\times_{I[1]} \mathrm{con}(V)$ be an arbitrary
element. Applying $\DDD$ gives the element
\begin{flalign}
\big(\LLL(L^\prime\, ) - \RRR(L),c\big) \in \mathrm{con}(\hom(V,W)) \subseteq \mathrm{end}(\hom(V,W))\times I[1]~,
\end{flalign}
where we regard $K := \LLL(L^\prime\, ) - \RRR(L)$  as a $k$-linear map
$K: \hom(V,W) \to \hom(V,W)$. We have to prove that
$K$ restricts to a $k$-linear map $K : \hom_A(V,W)\to \hom_A(V,W)$ on the $k$-submodules $\hom_A(V,W)\subseteq \hom(V,W)$ 
given in (\ref{eqn:hom}).
This amounts to showing that
\begin{flalign}\label{eqn:homcontmp}
\zeta([\,\cdot\,,\,\cdot\,])\big(K(M)\big) =0 \in \hom(A,\hom(V,W))~,
\end{flalign}
for all $M\in \hom_A(V,W)$. Let us for the moment consider the case of
trivial associator $\phi=1\otimes 1\otimes 1 $. Then the equality (\ref{eqn:homcontmp}) can be easily verified
by acting on generic homogeneous elements $a\in A$, which yields the equation
\begin{flalign}
\nn \left(\zeta([\,\cdot\,,\,\cdot\,])\big(K(M)\big)\right)(a) &=  \big[K(M),a\big]\\[4pt]
&=\big[L^\prime \bullet M - (-1)^{\vert M\vert\,\vert L\vert}\,(R^{(2)}\ra M)\bullet (R^{(1)}\ra L),a\big]=0~.
\end{flalign}
This equation follows from
\begin{subequations}
\begin{flalign}
\nn \big[L^\prime\bullet M,a\big] &= L^\prime \bullet \big[M,a \big] + 
(-1)^{\vert M\vert \,\vert a\vert}~\big[L^\prime, R^{(2)}\ra a \big]\bullet \big(R^{(1)}\ra M \big)\\[4pt]
\nn &= 0 +(-1)^{\vert M\vert \,\vert a\vert}~\big(\dd(c)\big)\big(R^{(2)}\ra a\big) \,\big(R^{(1)}\ra M \big)\\[4pt]
&= (-1)^{\vert L\vert\,\vert M\vert}~M\, \big(\dd(c)\big)(a)~,
\end{flalign}
where in the last equality we used $\vert L\vert = \vert L^\prime\, \vert$, and
\begin{flalign}
\nn \big[(R^{(2)}\ra M)\bullet (R^{(1)}\ra L),a\big] &= (R^{(2)}\ra M)\bullet \big[R^{(1)}\ra L ,a\big]\\
\nn &\quad + (-1)^{\vert L\vert \,\vert a\vert}~\big[R^{(2)}\ra M,\widetilde{R}{}^{(2)}\ra a \big]\bullet
\big((\widetilde{R}{}^{(1)}\, R^{(1)}) \ra L \big)\\[4pt]
&=M \big(\dd(c)\big)(a) +0~.
\end{flalign}
\end{subequations}
The equality (\ref{eqn:homcontmp}) is also true for the case of nontrivial associators, however
the corresponding calculation is much more lengthy and involved, and hence we will not write it out in detail.
\end{proof}

%%%%%%%%%%%%%%%%%%%%%%%%%%%%%%%%%%%%%%%%%%%%%%%%%%%%%%%
%%%%%%%%%%%%%%%%%%%%%%%%%%%%%%%%%%%%%%%%%%%%%%%%%%%%%%%

\subsection{Cochain twisting of connections}
The cochain twist deformation quantization functor preserves  connections.
\begin{propo}\label{propo:cochaincon}
Let $(A, \dd)$ be any differential calculus in $^H\MMM$, $V$ any
object in $ {}^H_{}{}^{}_{A}\MMM^{\mathrm{sym}}_{A} $ and $F$ any cochain twisting element based on $H$. 
Then the coherence map $\gamma \times \psi : \mathrm{end}_F(\FF(V))\times 
I_F[1] \to \FF(\mathrm{end}(V))\times \FF(I[1])$ restricts to an ${}^{H_F}_{}\MMM$-isomorphism
\begin{flalign}
\gamma \times \psi :  \mathrm{con}_F(\FF(V)) \longrightarrow \FF(\mathrm{con}(V))~.
\end{flalign}
\end{propo}
\begin{proof}
The proof follows that of Proposition~\ref{propo:cochainder} and 
it requires showing commutativity of the diagram
\begin{flalign}
\xymatrix{
\ar[dd]_-{\gamma\times \psi}\mathrm{end}_F(\FF(V)) \times I_F[1] \ar@<-1ex>[rrrrr]_-{\zeta_F(\widehat{l}_F\circ \ev_F \circ (\dd_F\otimes_F\id)\circ (\mathrm{pr}_2\otimes_F \id))} 
\ar@<1ex>[rrrrr]^-{\zeta_F([\,\cdot\,,\,\cdot\,]_F\circ (\mathrm{pr}_1\otimes_F \id))}   &&&&& \hom_F(\FF(A), \mathrm{end}_F(\FF(V)))\ar[d]^-{\gamma\circ (\,\cdot\,)}\\
&&&&& \hom_F(\FF(A), \FF(\mathrm{end}(V))) \ar[d]^-{\gamma}\\
\FF(\mathrm{end}(V)) \times \FF(I[1]) \ar@<-1ex>[rrrrr]_-{\FF(\zeta(\widehat{l}\circ \ev \circ (\dd\otimes\id)\circ (\mathrm{pr}_2\otimes \id)))} 
\ar@<1ex>[rrrrr]^-{\FF(\zeta([\,\cdot\,,\,\cdot\,]\circ (\mathrm{pr}_1\otimes \id)))} &&&&&\FF\big(\hom(A, \mathrm{end}(V))\big)
}
\end{flalign}
in ${}^{H_F}\MMM$, which is a straightforward but slightly lengthy calculation.
\end{proof}
\begin{rem}
Applying this result to  Example \ref{ex:Gequibundles},
we find in particular that on any cochain twist deformation 
of any $G$-equivariant vector bundle there exists at least one connection, because
there exist connections on classical $G$-equivariant vector bundles (in the smooth category).
\end{rem}

%%%%%%%%%%%%%%%%%%%%%%%%%%%%%%%%%%%%%%%%%%%%%%%%%%%%%%%
%%%%%%%%%%%%%%%%%%%%%%%%%%%%%%%%%%%%%%%%%%%%%%%%%%%%%%%

\section{\label{sec:curvature}Curvature}
We shall develop the notion of curvature of connections on objects in
${}^H_{}{}^{}_{A}\MMM^{\mathrm{sym}}_{A}$ and compute explicitly the
curvatures of tensor product connections given by our construction in Theorem~\ref{theo:sumcon}.
We conclude by giving a brief sketch of how our formalism
can be used to describe a noncommutative and nonassociative theory of gravity
that is based on Einstein-Cartan geometry.
Throughout this section we have to make the assumption that
$H$ is a triangular quasi-Hopf algebra, which is in particular satisfied for our main examples of interest, see
Examples~\ref{ex:classicalGMan}, \ref{ex:Gequibundles} and \ref{ex:classicaldiffcalc}.

%%%%%%%%%%%%%%%%%%%%%%%%%%%%%%%%%%%%%%%%%%%%%%%%%%%%%%%
%%%%%%%%%%%%%%%%%%%%%%%%%%%%%%%%%%%%%%%%%%%%%%%%%%%%%%%

\subsection{Definition and properties}
For any object $V$ in ${}^H_{}{}^{}_{A}\MMM^{\mathrm{sym}}_{A}$, we define the ${}^H\MMM$-morphism
\begin{flalign}
\nn [\![\, \cdot\, , \, \cdot\, ]\!]: (\mathrm{end}(V) \times I[1]) \otimes (\mathrm{end}(V) \times I[1]) &\longrightarrow \mathrm{end}(V)~,\\
(L,c) \otimes (L^\prime, c^\prime\, ) &\longmapsto [L,L^\prime\, ]~.\label{eqn:doublebracket}
\end{flalign}

\begin{lem}\label{lem:bracketforcurv}
Let $H $ be a triangular quasi-Hopf algebra. Let
$(A, \dd)$ be a differential calculus in $^H\MMM$ and $V$ any object in ${}^H_{}{}^{}_{A}\MMM^{\mathrm{sym}}_{A}$. 
Then (\ref{eqn:doublebracket}) restricts to an ${}^H\MMM$-morphism
\begin{flalign}
[\![\, \cdot \, , \, \cdot\, ]\!]: \mathrm{con}(V) \otimes \mathrm{con}(V)&\longrightarrow \mathrm{end}_A(V)~.
\end{flalign}
\end{lem}
\begin{proof}
By Lemma~\ref{lem:homA} it is sufficient to show that
\begin{flalign}
\big[\big[\!\big[(L,c),(L^\prime, c^\prime\, )\big]\!\big] \,,\,a\big] = \big[[L,L^\prime\, ],a\big]= 0~,
\end{flalign}
for all homogeneous $(L,c),(L^\prime,c^\prime\, )\in\mathrm{con}(V)$ and $ a \in A $. 
Using the braided Jacobi identity and braided antisymmetry of Proposition~\ref{prop:bracket} 
(this is where we need triangularity) we obtain
\begin{flalign}
\nn\big[[L,L^\prime\, ],a\big] &= - (-1)^{\vert L\vert \,(\vert L^\prime\vert + \vert a\vert)}\, 
\big[\big[R^{(2)}_{(1)} \, \phi^{(2)}\ra L^\prime , R^{(2)}_{(2)}\, \phi^{(3)}\ra a \big]\,,\, R^{(1)}\, \phi^{(1)}\ra L\big] \\
\nn &\quad\quad + (-1)^{\vert a\vert \, \vert L^\prime\vert}\, 
 \big[\big[\widetilde{R}{}^{(2)}\, \phi^{(-2)}\, R^{(1)}_{(1)}\ra L , \widetilde{R}{}^{(1)}\, \phi^{(-1)}\, R^{(2)}\ra a\big] \,,\, 
\phi^{(-3)}\, R^{(1)}_{(2)}\ra L^{\prime}\, \big]\\[4pt]
\nn &= - (-1)^{\vert L\vert \,(\vert L^\prime\vert + \vert a\vert)}\,  \big[\ev\big(\dd(c^\prime\, )\otimes (R^{(2)}\ra a)\big)\,,\, R^{(1)}\ra L \big]\\
\nn &\quad\quad + (-1)^{\vert a\vert \, \vert L^\prime\vert}\, \big[\ev\big(\dd(c)\otimes (R^{(2)}\ra a)\big)\,,\, R^{(1)}\ra L^\prime\, \big]\\[4pt]
\nn &=\big[L , \ev\big(\dd(c^\prime \, )\otimes a\big)\big] - (-1)^{\vert L\vert \,\vert L^\prime\vert} ~
\big[L^\prime, \ev\big(\dd(c)\otimes a\big)\big]\\[4pt]
\nn &=\ev\big(\dd(c)\otimes  \ev\big(\dd(c^\prime\, )\otimes a\big)\big) - (-1)^{\vert L\vert \,\vert L^\prime\vert} ~
\ev\big(\dd(c^\prime\, )\otimes \ev\big(\dd(c)\otimes a\big)\big)\\[4pt]
&=\ev\big(\big(\dd(c)\bullet\dd(c^\prime\, )\big)\otimes a\big)- (-1)^{\vert L\vert \,\vert L^\prime\vert} ~
\ev\big(\big(\dd(c^\prime\, )\bullet\dd(c)\big)\otimes a\big) =0~,
\end{flalign}
where we have also used Lemma~\ref{lem:con},
Lemma~\ref{lem:compproperties} together with the normalization
$(\epsilon\otimes \epsilon \otimes\id)(\phi)=1$ of the associator,
and nilpotency of $\dd : I[1] \to\mathrm{der}(A) $ from Definition \ref{defi:diffcalc}.
\end{proof}

With these techniques we can now define the curvature of a connection.
Since the curvature is supposed to be quadratic in the connections, we cannot realize the assignment
of curvatures as an ${}^H\MMM$-morphism. We shall employ the following element-wise definition.
\begin{defi}\label{def:curv}
Let $ (A, \dd) $ be a differential calculus in $^H\MMM$ and let $ V $ be an
object in  $ {}^H_{}{}^{}_{A}\MMM^{\mathrm{sym}}_{A} $. 
The curvature of a connection $\nabla := (L,1)\in \mathrm{con}(V)$ is the element
\begin{flalign}\label{eqn:curv}
\mathrm{Curv}(\nabla) :=  \big[\!\big[\nabla,\nabla\big]\!\big]\in \mathrm{end}_A(V)~.
\end{flalign}
\end{defi}
\begin{rem}
Given any connection $\nabla := (L,1)\in \mathrm{con}(V)$, we can define the
Bianchi tensor corresponding to $\nabla$ as
\begin{flalign}
\mathrm{Bianchi}(\nabla) := \ev \big(\DDD(\nabla, \nabla) \otimes \mathrm{Curv}(\nabla) \big)\in \mathrm{end}_A(V)~.
\end{flalign}
In contrast to the situation in classical differential geometry, here the Bianchi tensor in general does not vanish.
Hence, it may be interpreted as a measure of the noncommutativity and nonassociativity of $A$, 
$V$ and $\nabla$.
\end{rem}

Finally, we observe an additive property of the curvature of
the tensor product connections constructed in Theorem~\ref{theo:sumcon}.
\begin{propo}
Let $H$ be a triangular quasi-Hopf algebra, $ (A, \dd)$ a differential calculus in $^H\MMM$ 
and $ V, W $ two objects in $ {}^H_{}{}^{}_{A}\MMM^{\mathrm{sym}}_{A} $. 
Given any two connections $\nabla_V := (L,1)\in\mathrm{con}(V)$ and $\nabla_W := (L^\prime,1)
\in\mathrm{con}(W)$, the curvature of their sum satisfies
\begin{flalign}
\mathrm{Curv}(\nabla_V \obulplus \nabla_W) = 
\mathrm{Curv}(\nabla_V) \obultimes 1 + 1 \obultimes \mathrm{Curv}(\nabla_W) ~.
\end{flalign}
\end{propo}
\begin{proof}
The proof follows from a simple calculation
\begin{flalign}
\nn \mathrm{Curv}(\nabla_V \obulplus \nabla_W) &= [L \obultimes 1 + 1 \obultimes L^\prime, 
L \obultimes 1 + 1 \obultimes L^\prime\, ]\\[4pt]
\nn &= [L,L] \obultimes 1 + 1 \obultimes [L^\prime,L^\prime\, ]\\[4pt]
&= \mathrm{Curv}(\nabla_V) \obultimes 1 + 1 \obultimes \mathrm{Curv}(\nabla_W) ~,
\end{flalign}
where we have used the properties in Lemma~\ref{lem:tensorproductproperties} and the
 braided antisymmetry of the internal commutator from Proposition~\ref{prop:bracket}~(i).
\end{proof}

%%%%%%%%%%%%%%%%%%%%%%%%%%%%%%%%%%%%%%%%%%%%%%%%%%%%%%%
%%%%%%%%%%%%%%%%%%%%%%%%%%%%%%%%%%%%%%%%%%%%%%%%%%%%%%%

\subsection{Einstein-Cartan geometry}
We conclude with a brief sketch of how our formalism
can be used to describe a noncommutative and nonassociative 
theory of gravity coupled to Dirac fields. Our considerations are based on Einstein-Cartan geometry
and its generalization to noncommutative geometry which was developed in~\cite{Catellani:2009}. 
Our strategy is to formulate classical Einstein-Cartan geometry in our abstract language
and then to give an outline of its cochain twist deformation quantization, which will lead
to a noncommutative and nonassociative gravity theory.
\sk

Let $M$ be any parallelizable manifold of dimension $m$. Associated
to $M$ is the Hopf algebra $H = U\, \mathrm{Vec}(M)$ (with trivial associator $\phi = 1\otimes 1\otimes 1$)
given by the universal enveloping
algebra of the Lie algebra of vector fields on $M$, which is triangular with trivial  $R$-matrix $R=1\otimes 1$.
We take $A := \Omega^\sharp(M)$ to be the exterior algebra 
of differential forms on $M$ and $V := \Omega^\sharp(M,S)$ to be the $A$-bimodule 
of spinor-valued differential forms. Then $A$ is an object in ${}^H\AAA^{\mathrm{com}}$
and $V$ is an object in $ {}^H_{}{}^{}_{A}\MMM^{\mathrm{sym}}_{A} $ 
for our choice of triangular Hopf algebra $H$. A vielbein is an invertible element $E \in  \mathrm{end}_A(V)\simeq 
\Omega^\sharp(M,\mathrm{end}(S))$ of the form 
(here and in the following summations over repeated pairs of indices are understood)
\begin{flalign}
E = E^a \,\gamma_a \ ,
\end{flalign}
where $E^a \in \Omega^1(M)$ are the components of the vielbein,
$\gamma_a$ are the gamma-matrices associated with the spin representation $S$
and  $a = 1,\dots, m$.  A spin connection is a connection $\nabla:= (L,1) \in \mathrm{con}(V)$ 
of the form
\begin{flalign}
\nabla = \big(\dd - \tfrac{1}{2}\, \omega^{ab} \, [\gamma_a,\gamma_b] \,,\, 1\big)~,
\end{flalign}
where $\omega^{ab} \in \Omega^1(M)$ are the components of the spin connection 
and $\dd$ is the exterior derivative. 
A Dirac field is an element $\psi \in V_0$ of $\bbZ$-degree 0 and 
a conjugate Dirac field is an element 
$\overline{\psi} \in {V^\vee}_0 := {\hom_A(V,A)}_0$ in the dual module of $\bbZ$-degree 0. 
Using Theorem~\ref{theo;adjointcon} we can induce a connection on conjugate Dirac fields by taking
\begin{flalign}
\nabla^\vee := \DDD\big(\nabla\,,\, (\dd,1)\big)\in\mathrm{con}(V^\vee)~.
\end{flalign}
Since $V$ is a free $A$-module there is an isomorphism
$V \otimes_A V^\vee \simeq \mathrm{end}_A(V)$, which we shall always suppress
from our notation.
\sk

We can now write down the Lagrangian of Einstein-Cartan gravity coupled to a Dirac field
in $ m $ dimensions within our formalism as the top form
\begin{flalign}\label{eqn:tracedlagrangian}
L^{(m)} := \mathrm{Tr} \big(\mathcal{L}^{(m)}\big) \in A_{m} = \Omega^m(M)~,
\end{flalign}
where 
\begin{multline}\label{eqn:lagrangian}
\mathcal{L}^{(m)} := \sqrt{-1} \ \mathrm{Curv}(\nabla) \bullet \underbrace{E \bullet E \bullet \cdots \bullet E}_\text{$m{-}2$ times} \bullet \gamma_5 \\[4pt]
-\Big(\ev\big(\nabla \otimes \psi \big) \otimes_A \overline{\psi} - \psi \otimes_A \ev\big(\nabla^\vee \otimes \overline{\psi} \, \big)\Big)
\bullet \underbrace{E \bullet E \bullet \cdots \bullet E}_\text{$m{-}1$ times} \bullet \gamma_5~
\end{multline}
and $ \gamma_5\in\mathrm{end}_A(V)$ (of $\bbZ$-degree $0$) 
is absent when $m$ is odd. The trace in (\ref{eqn:tracedlagrangian}) 
is the ${}^H\MMM$-morphism
\begin{flalign}\label{eqn:trace}
\mathrm{Tr}: \mathrm{end}_A(V) \longrightarrow A~
\end{flalign}
given by the pointwise trace in $\mathrm{end}_A(V) \simeq \Omega^\sharp(M, \mathrm{end}(S))$.
\sk

To obtain a noncommutative and nonassociative theory of gravity
one can quantize $A$ and $V$ by a suitable cochain twist $F$, 
see e.g.\ \cite{Mylonasb,Aschieri:2015roa} and~\cite[Example 6.5]{Barnes:2014} for the explicit examples
of relevance to the string theory applications mentioned in Section~\ref{sec:intro}. The construction of the Lagrangian (\ref{eqn:tracedlagrangian})
then proceeds in the same way as above, but now with all internal constructions made in the
braided closed monoidal category ${}^{H_F}\MMM$ rather than in ${}^H\MMM$. 
As cochain twisting may induce a nontrivial associator in ${}^{H_F}\MMM$, our definition of 
the Lagrangian (\ref{eqn:tracedlagrangian}) has to be supplemented with a bracketing convention
for the internal compositions in \eqref{eqn:lagrangian}. The study of which choice of bracketing leads to a physically reasonable
model for noncommutative and nonassociative gravity is beyond the scope of the present paper
and will be addressed in future work.

%%%%%%%%%%%%%%%%%%%%%%%%%%%%%%%%%%%%%%%%%%%%%%%%
%%%%%%%%%%%%%%%%%%%%%%%%%%%%%%%%%%%%%%%%%%%%%%%%

\section*{Acknowledgements}
We thank Paolo Aschieri and Edwin Beggs for helpful discussions. This work was supported in part by the Action 
MP1405 QSPACE from the European Cooperation in Science and Technology
(COST). G.E.B.\ is a Commonwealth Scholar, funded by the UK government. 
The work of A.S.\ is supported by a Research Fellowship of the Deutsche Forschungsgemeinschaft (DFG, Germany). 
The work of R.J.S.\ is supported in part by the Consolidated Grant ST/L000334/1 
from the UK Science and Technology Facilities Council (STFC). 

%%%%%%%%%%%%%%%%%%%%%%%%%%%%%%%%%%%%%%%%%%%%%%%%
%%%%%%%%%%%%%%%%%%%%%%%%%%%%%%%%%%%%%%%%%%%%%%%%

%%%%%%%%%%%%%%%%%%%%%%%%%%%%%%%%%%%%%%%%%%%%%%%%
%%%%%%%%%%%%%%%%%%%%%%%%%%%%%%%%%%%%%%%%%%%%%%%%

\end{document}